\documentclass[12pt]{amsart}

\usepackage{amsmath, graphicx, cite, enumerate, comment}
\usepackage{amssymb, amsthm, amscd}
\usepackage{latexsym}
\usepackage{amssymb}
\usepackage[T1]{fontenc}	
\usepackage[english]{babel}
\usepackage{amsmath,amssymb}
\usepackage{amsfonts}
\usepackage{amscd}
\usepackage{cite}
\usepackage{marvosym}
\usepackage{mathrsfs}
\usepackage{amsthm}
\usepackage{hyperref}
\usepackage{verbatim}

\usepackage{hyperref}
\hypersetup{
	bookmarks=true,         % show bookmarks bar?
	unicode=false,          % non-Latin characters in AcrobatÃÂ¢ÃÂÃÂs bookmarks
	pdftoolbar=true,        % show AcrobatÃÂ¢ÃÂÃÂs toolbar?
	pdfmenubar=true,        % show AcrobatÃÂ¢ÃÂÃÂs menu?
	pdffitwindow=false,     % window fit to page when opened
	pdfstartview={FitH},    % fits the width of the page to the window
	pdftitle={My title},    % title
	pdfauthor={Author},     % author
	pdfsubject={Subject},   % subject of the document
	pdfcreator={Creator},   % creator of the document
	pdfproducer={Producer}, % producer of the document
	pdfkeywords={keyword1, key2, key3}, % list of keywords
	pdfnewwindow=true,      % links in new PDF window
	colorlinks=true,       % false: boxed links; true: colored links
	linkcolor=black,          % color of internal links (change box color with linkbordercolor)
	citecolor=blue,        % color of links to bibliography
	filecolor=magenta,      % color of file links
	urlcolor=cyan           % color of external links
}

\usepackage{tikz} \usetikzlibrary{arrows,shapes,patterns} % AC: You need this command for everything below to work
\usepackage{lipsum} % for sample

\newtheorem{theorem}{Theorem}[section]

\newtheorem{lemma}[theorem]{Lemma}
\newtheorem{proposition}[theorem]{Proposition}

\theoremstyle{definition}

\newtheorem{example}[theorem]{Example}

\theoremstyle{remark}
\newtheorem{remark}[theorem]{Remark}

\theoremstyle{remark}

\theoremstyle{definition}

\numberwithin{equation}{section}

\newcommand{\R}{\mathbb{R}}
\newcommand{\de}{\delta}

\newcommand{\s}{\sigma}

\newcommand{\cL}{\mathcal{L}}

\newcommand{\cQ}{\mathcal{Q}}

\newcommand{\be}{\begin{equation}}
\newcommand{\ee}{\end{equation}}
\newcommand{\bd}{\begin{displaymath}}
\newcommand{\ed}{\end{displaymath}}

\begin{document}
            
            	\thanks{\textsl{Mathematics Subject Classification (MSC 2010):} Primary 53A10; Secondary 83C99, 49Q05.}
            	\title[Non-existence of minimal catenoids in af spaces]
            	{A non-existence result for minimal catenoids in asymptotically flat spaces}
            	\author{Alessandro Carlotto}
            	\address{ETH - Institute for Theoretical Studies  \\
            		ETH \\
            		Z\"urich, Switzerland}
            	\email{alessandro.carlotto@eth-its.ethz.ch}
            	\author{Andrea Mondino}
            	\address{Institut f\"ur Mathematik \\
            		Universit\"at Z\"urich \\
            		Z\"urich, Switzerland}
            	
            	\email{andrea.mondino@math.uzh.ch}           
            	
            	\begin{abstract} We show that asymptotically Schwarzschildean 3-manifolds cannot contain minimal surfaces obtained by perturbative deformations of a Euclidean catenoid, no matter how small the ADM mass of the ambient space and how large the neck of the catenoid itself. Such an obstruction is sharply three-dimensional and ceases to hold for more general classes of asymptotically flat data.
            	\end{abstract}
            	
            	\maketitle
            	
            	   	\section{Introduction}
            	
            	In general relativity, asymptotically flat manifolds naturally arise as spacelike slices in space-times describing isolated gravitational systems. Due to their self-evident physical relevance, their geometry has been extensively studied and a variety of fundamental results have been obtained. Among these, the mean-curvature proof of the \textsl{Positive Mass Theorem} proposed by Schoen and Yau \cite{SY79} has disclosed a fundamental principle, which basically asserts that an asymptotically Schwarzschildean manifold (of non-negative scalar curvature, as prescribed by the dominant energy condition) cannot contain an asymptotically planar stable minimal surface. In fact, much more recently, the following non-existence result has been obtained by the first-named author:
            	
            	\begin{theorem}
            		\label{thm:ne1}\cite{Car14a}	
            		Let $(M,g)$ be an asymptotically Schwarzschildean 3-manifold of non-negative scalar curvature. If it contains a complete, properly embedded stable minimal surface $\Sigma$, then $\left(M,g\right)$ is isometric to the Euclidean space $\R^{3}$ and $\Sigma$ is an affine plane. 	
            	\end{theorem}	  
            	
            	We refer readers to Subsection \ref{subs:as} for a precise definition of asymptotically Schwarzschildean spaces, but for the sake of this introduction they may simply consider them a suitably isotropic subclass of asymptotically flat manifolds. \\
            	\indent It is worth noting that the properness assumption is inessential (and thus can be removed) when $(M,g)$ has an horizon boundary, as was shown in \cite{CCE15}. Even more importantly, it is appropriate to recall from \cite{EM12} that these rigidity results for stable minimal surfaces have important consequences on the behaviour of isoperimetric domains for large enclosed volume (see \cite{EM13a, EM13b} and \cite{MFN15} for a general picture of the isoperimetric structure of asymptotically Schwarzschildean spaces, see also \cite{MN12} for isoperimetric properties of locally asymptotically flat spaces of non-negative Ricci curvature). In fact, using these tools it was shown in Corollary 1.11 in \cite{CCE15} that (under a natural non-drifting assumption) any large volume-preserving stable constant mean curvature sphere has to be a leaf of the canonical Huisken-Yau foliation \cite{HY96}, thus essentially a perturbed sphere centered at the center of mass of the ambient manifold. Lastly, theorems of this type can be used to classify asymptotically flat static manifolds (Corollary 1.9 in \cite{CCE15}) by virtue of the area-minimizing properties of the zero set of the corresponding potentials (see \cite{GM14}). \\

            	   	\indent From a different perspective, the theorem above ensures that (in this setting) the most basic mechanism to generate complete (non-compact) minimal surfaces, namely taking limits of solutions of the Plateau problem for diverging boundaries, is inevitably doomed to fail. It is then a natural question whether some other constructive procedure has chances to succeed, or more generally whether asymptotically Schwarzschildean manifolds contain (necessarily unstable) unbounded minimal surfaces. \\

            	   	\indent In trying to prove an existence result via perturbation methods, there is a class of Euclidean minimal surfaces which are by far the most natural to deform, namely catenoids. The reason is twofold: on the one hand such surfaces blow-down to a double plane (so that, in asymptotically flat spaces, they have very small mean curvature far away from the origin) and, perhaps more importantly, they come in one-parameter families parametrized by the neck-parameter, so that choosing that parameter to be sufficiently large we can work in a region where the ambient metric is uniformly close to flat. In other words, the effect of gravity can be made arbitrarily small. Hence, one can try to construct minimal catenoids in asymptotically Schwarzschildean spaces by first performing a Lyapunov-Schmidt reduction (thereby constructing normal graphs over Euclidean catenoids whose defining functions belong to suitable weighted Sobolev spaces) and then handling the cokernel, generated by Euclidean isometries and by the scaling parameter, by exploiting the freedom on the choice of the center, of the symmetry axis and of the radius of the the neck itself. In this article, we prove that this approach is also doomed to fail, and that in fact the ADM mass obstructs the existence of minimal catenoids, even when the function defining the normal perturbation is (potentially) allowed to decay very slowly, the mass is very small and the neck is very large. \\

            	   	Let then $(M^3,g^{m,e_m})$ be a one-parameter family of asymptotically Schwarzschildean manifolds\footnote{Our non-existence result is patently local with respect to a given end of the ambient manifold $M$, hence we can (and we shall) assume that $M$ has only one end, and in fact we can reduce to the case of a one-parameter family of asymptotically Schwarzschildean metrics on $\mathbb{R}^3\setminus B_{r_0}$.} (see Subsection \ref{subs:as}) such that $g^{0,e_0}$ is the flat Euclidean metric on some exterior domain of $\mathbb{R}^3$ and assume, correspondingly, that $(S_m)$ is a $C^1-$arc of surfaces emanating from the catenoid $\textbf{Cat}_c$ of neck $c\geq c_0$. When proving a non-existence result for an arc of \textsl{minimal} surfaces, we need to make sure that the asymptotic conditions we prescribe on each $S_m$ are not too restrictive and do not determine fictitious obstructions. Significantly enough, our main theorem concerns the class of surfaces that can be parametrized as normal graphs (over some Euclidean catenoid) by means of a defining function $\Omega$ which is solely required to be $C^2-$bounded in the cylindrical coordinates $(u,v)\in S^1\times\mathbb{R}$ (see Subsection \ref{subs:catgr}) and such that $\Omega(u,v)\to 0$ as one lets $v\to\pm\infty$. Hence, \textsl{no a priori assumption on the rate of decay  is made}. \\

            	    \indent We let $F_c:S^1\times\mathbb{R}\to\mathbb{R}^3$ denote the standard local parametrization of the catenoid of neck $c$ and vertical axis (see equation \eqref{eq:catc}), by $\textbf{n}$ the unit normal vector field defined by \eqref{eq:defn} and we set $F^{\Omega}_{c}=F_c+\textbf{n}\Omega$ (the normal graph over the catenoid of neck $c$ whose defining function is $\Omega$, see equation \eqref{eq:norgraph}).
            	    Lastly, the precise definition of the functional spaces $\mathscr{E}^k$ and $\mathscr{C}^2_{b}$ is provided in Section \ref{sec:setup}.\\

            	    \indent Here is the statement of our main result:

            	  \begin{theorem}\label{thm:main}
            	   	 		Given $r_0>0$ and $e_{(\cdot)}:[0,\varepsilon_0]\to \mathscr{E}^2$ a continuous curve differentiable at $0$ with $e_0=0$, let us consider the manifold with boundary $M=\R^3 \setminus B_{r_0}$ endowed with the one-parameter family of asymptotically Schwarzschildean metrics 
            	   	 		\[
            	   	 		g^{m,e_m}_{ij}:=\left(1+\frac{m}{2|x|} \right)^4 \delta_{ij}+(e_m(x))_{ij}.
            	   	 		\]
            	   	 		Then there exists $\bar{c}\geq r_0$ for which the following conclusion holds: it is not possible to find 
            	   	 		\begin{itemize}
            	   	 			\item a continuous curve $\Omega_{(\cdot)}:[0,\varepsilon_0] \to \mathscr{C}^2_b(S^1\times \R)$  differentiable at $0$, with $\Omega_0=0$, and satisfying $\Omega_{m}(u,v)=o(1)$ as $v\to\pm\infty$ for each $m\in [0, \varepsilon_0]$,
            	   	 			\item sequences $m_i\downarrow 0$, $0<m_i<\varepsilon_0$,   and $c_i=c(m_i)\geq \bar{c}$,
            	   	 		\end{itemize}
            	   	 		such that  the perturbed catenoidal immersions $F^{\Omega_{m_i}}_{c_i}$ are  minimal in $(\R^3 \setminus B_{r_0}, g^{m_i,e_{m_i}})$.
            	   	 \end{theorem}	 
            
            		We shall add three important remarks to the statement of this theorem. 
            		
            		\begin{remark}\label{rem:higherdim}
            				Such an obstruction ceases to hold when the ambient manifold has dimension at least four. For instance, if one considers radially symmetric
            				asymptotically flat manifolds it is possible to show (by a rather direct application of the implicit function theorem) that given any minimal $n$-catenoid in $\mathbb{R}^{n+1}$ (with $n\geq 3$) there exists a family of normal perturbations whose defining function is exponentially decreasing along the two ends and whose mean curvature is zero (see Proposition \ref{pro:highdim} for a precise statement). For the sake of completeness a proof of this assertion is provided in Section \ref{sec:app}. Notice that this result does not involve any curvature condition. We certainly do not expect such symmetry assumption to be necessary in order to prove a perturbative existence theorem and have made it for the sole sake of simplicity, since our result suffices for the purposes of showing the absence of obstructions in the higher-dimensional scenario.
            			            		\end{remark}	
            		
            		\begin{remark}\label{rem:exotic}
            				When one allows the ambient metric \textsl{not} to be asymptotically Schwarzschildean (thereby allowing highly an-isotropic asymptotics) the conclusion of Theorem \ref{thm:main} is no longer true. Indeed, the first-named author and R. Schoen constructed in \cite{CS14} asymptotically flat manifolds that have positive ADM mass but are \textsl{exactly flat} on a half-space, with the immediate consequence that such half-space is foliated by stable minimal surfaces. Furthermore, the same general localization scheme allows to produce data which are non-trivial outside a cone of small aperture, and hence those manifolds will contain \textsl{Euclidean} minimal catenoids so that in the statement of Theorem \ref{thm:main} the proviso of Schwarzschildean asymptotics is certainly necessary for the conclusion to hold. Further comments concerning the precise rate of decay of the ambient metric and the range of validity of our main theorem are provided in Subsection \ref{subs:as} (see, more specifically, Remark \ref{rem:ratedecay}).
            		\end{remark}

            		\begin{remark}\label{rem:ScalNonNeg}
            			By the  rigidity in the Positive Mass  Theorem of Schoen-Yau \cite{SY79}, if we assume  $e_0\in \mathcal{E}^2$ to be such that  the metric $g^{0,e_0}_{ij}:=\delta_{ij}+(e_0(x))_{ij}$ has non negative scalar curvature, then necessarily  $e_0=0$.  We remind the reader that the assumption $\textrm{Scal}\geq 0$ is rather customary in general relativity as it corresponds to the dominant energy condition for time-symmetric data, so that the assumption $e_0=0$ in Theorem \ref{thm:main} is  very natural.
            		\end{remark}
            		
            		Furthermore, it is important to observe that Theorem \ref{thm:main} obstructs not only the construction of catenoidal minimal surfaces by direct applications of the Implicit Function Theorem, but also (and more generally) by a Lyapunov-Schmidt reduction argument (that is to say: handling the cokernel). Indeed, suppose that for $e_{(\cdot)}:[0,\varepsilon_0]\to \mathscr{E}^2$  as above there exist three continuous maps:
            		\[
            		\Omega_{(\cdot)}:[0,\varepsilon_0] \to \mathscr{C}^2_b(S^1\times \R), \ \zeta: [0,\varepsilon_0]\to\mathbb{R}^3, \rho: [0,\varepsilon_0]\to SO(2)
            		\]
            		such that (in fixed coordinates $\left\{x\right\}$) the normal graph defined by $\Omega_m$ over the catenoid of center $\zeta_m$ and axis the image via $\rho_m$ of the $x^3-$axis is minimal. Then one could simply consider the new family of metrics
            		\[
            		\tilde{g}^{m,e_m}(y):=g^{m,e_m}(\zeta_m+\rho_m(y))
            		\] 
            		which are readily checked to be of the form
            		\[
            		\tilde{g}_{ij}^{m,e_m}(y)=\left(1+\frac{m}{2|y|}\right)^4\delta_{ij}+\tilde{e}_m(y)_{ij}
            		\]
            		with $\tilde{e}_m\in\mathscr{E}^2$ and hence directly apply Theorem \ref{thm:main} to get a contradiction.
            		Of course, we have tacitly used here the fact that the Jacobi fields of $\textbf{Cat}_c$ are those generated by Euclidean isometries plus the dilations (see Section \ref{sec:app}) and notice that our non-existence theorem allows for some freedom in the choice of the neck of the catenoid in question.\\

            		\indent As a result, Theorem \ref{thm:main} asserts the failure of a wide class of perturbative methods to \textsl{minimally} deform Euclidean catenoids, such conclusion being rather sharp both with respect to the asymptotic structure of the ambient Riemannian manifold and with respect to the dimension of the manifold itself. This result suggests that the landscape of minimal surfaces in asymptotically Schwarzschildean spaces differs dramatically from its Euclidean counterpart, with much more severe restrictions and limitations that are still to be fully understood.
            		
            	\

               	\textsl{Acknowledgments}. The first-named author would like to thank Richard Schoen for a number of enlightening conversations on themes related to the large-scale structure of asymptotically flat spaces and minimal surfaces therein, Rafe Mazzeo for important clarifications and Jan Metzger for his interest in this work. The authors also wish to express their sincere gratitude to the anonymous referee for carefully reading this article and for preparing a detailed report, which resulted in a significantly improved version of the paper.\\
            	\indent This article was done while  A. C. was an ETH-ITS fellow: the outstanding support of Dr. Max R\"ossler, of the Walter Haefner Foundation and of the ETH Z\"urich Foundation are gratefully acknowledged. \\
            	\indent Part  of the work has been  developed while A. M. was lecturer at the  Institut f\"ur Mathematik at the  Universit\"at Z\"urich and the project was finalized while he was in residence at the Mathematical Science Research Institute in Berkeley, during the spring semester 2016, supported by the National Science Foundation under the Grant DMS-1440140. He wishes to express his gratitude to both institutes for the stimulating atmosphere and the excellent working conditions. 
            	
            	\
            	\
            	
            	\section{Setting and recollections}\label{sec:setup}
            	
            	\subsection{Asymptotically Schwarzschildean data} \label{subs:as}
            	
            		We shall say that a complete, Riemannian $(n+1)$-manifold $(M,g)$ is (one-ended) asymptotically Schwarzschildean if there exists a compact set $K\subset M$ such that $M\setminus K$ is diffeomorphic to $\mathbb{R}^{n+1}\setminus B_{r_{0}}$ (for some constant $r_0>0$) and in the coordinate system $\left\{x\right\}$ induced by such diffeomorphism the metric takes the form
            		\begin{equation}\label{eq:defgme}
            		g^{m,e}_{ij}(x)=\left(1+\frac{m}{2|x|^{n-1}} \right)^{\frac{4}{n-1}}\delta_{ij} +(e(x))_{ij}
            		\end{equation}
            		where the tensor $e(x)$ is assumed to satisfy the bounds
            		\be\label{eq:errorterm}
            		|e(x)|=O(|x|^{-n}), \ \ |\partial e(x)|=O(|x|^{-n-1}), \ \ldots \ |\partial^k e(x)|=O(|x|^{-n-k}) 
            		\ee
            		which we shall summarize by simply writing $e(x)=O_k(|x|^{-n})$; more precisely we shall require $e$ to be an element of the Banach space
            		\[
            		\mathscr{E}^{k}:=\left\{e:\mathbb{R}^{n+1}\setminus B_{r_0}\to\mathcal{M}^{3\times 3}_{sym}(\mathbb{R}), \ \|e\|_{\mathscr{E}^k}:=\sum_{i,j=1}^{n+1}\sum_{|\alpha|\leq k} \sup_{|x|\geq r_0}|\partial^{\alpha}e_{ij}(x)||x|^{n+|\alpha|}<\infty    \right\}.
            		\]
            		We are always tacitly assuming $n\geq 2$. \\

            		\indent  We remark that the number $m$ equals the ADM mass of the asymptotically flat manifold in question (see \cite{ADM59, Bar86} and \cite{Wal84} for a broader contextualization). Also, recall that asymptotically Schwarzschildean data are dense in general asymptotically flat data with respect to suitably weighted Sobolev (or Schauder) norms, see \cite{EHLS11} and reference therein. 
            		
            		\begin{remark}\label{rem:ratedecay}
            		The whole proof of Theorem \ref{thm:main} would in fact go through under the milder assumption that the error term satisfies $|e(x)|=O_2(|x|^{-1-\tau})$ for \textsl{some} $\tau>0$, with complications of purely notational character. Instead, we shall remark that the assumption $|e(x)|=o_2(|x|^{-1})$ does not seem sufficient for the argument to go through (the obstruction being related to the proof of Lemma \ref{lem:expansion}). However, if one \textsl{postulates} the defining functions $\Omega_{\left(\cdot\right)}$ to have a sufficiently good rate of decay (enough to legitimate integration by parts in Lemma \ref{lem:IntParts}) then the argument presented in Subsection \ref{subs:endproof} can be applied.
            			\end{remark}

            	\subsection{Catenoidal graphs} \label{subs:catgr}
            	
            		\indent Given $n\ge 2$, let us recall that  the minimal $n$-catenoid  $\textbf{Cat}_c$ of neck $c>0$ in $\mathbb{R}^{n+1}$ (with vertical axis)  can be parametrized by means of the smooth map $(\theta,s)\in S^{n-1}\times \R \to \R^{n+1}$ defined by $ F_c(\theta,s):=c(\phi(s) \theta, \psi(s))$ where
            		\[
            		\phi(s):=(\cosh((n-1) s))^{\frac{1} {n-1}}, \quad \psi(s):= \int_0^s \phi^{2-n}(\sigma) \, d \sigma.
            		\]
            		By direct computation one can check that the Jacobi operator of $\textbf{Cat}_c$ (i.e. the linearized mean curvature operator) is given by 
            		\[L_{F_c}^{\de}:=\frac{1}{c^2}\left(\frac{1}{\phi^n} \frac{\partial}{\partial s} \left( \phi^{n-2} \frac{\partial}{\partial s} \right) + \frac{1}{\phi^2} \Delta_{S^{n-1}} + \frac{n(n-1)}{\phi^{2n}}\right). 
            		\]
            		
            		A choice of unit normal vector to $\textbf{Cat}_c$ is given by
            		\[
            		\textbf{n}_{F_c}^{\de}=\phi(s)^{-1}(\phi^{2-n}(s)\theta,-\dot{\phi}(s)).
            		\]

            		Given an integer $k\geq 0$ and $q\in\mathbb{R}$, we shall then introduce the weighted functional spaces $\mathscr{C}^k_q(S^{n-1}\times\mathbb{R})$ on the $n-$dimensional cylinder that consist of those $\mathscr{C}^k$ functions for which the norm
            		\[
            		\left\|f\right\|_{\mathscr{C}^{k}_q}=\sum_{|\alpha|\leq k}\sup_{(\theta,s)\in S^{n-1}\times\mathbb{R}}|\nabla^{\alpha}f(\theta,s)||\phi^{-q}(s)|
            		\]
            		is finite.
            		In the special case $q=0$ we shall denote the corresponding space by $\mathscr{C}^k_b(S^{n-1}\times\mathbb{R})$. \\
            		\indent Furthermore, we will consider the class of functions $\mathscr{C}^{k}_{(-1)^+}(S^{n-1}\times\mathbb{R})$ given by the intersection
            		\[
            		\mathscr{C}^{k}_{(-1)^+}(S^{n-1}\times\mathbb{R})=\bigcap_{q\in (-1,0)}	\mathscr{C}^{k}_{q}(S^{n-1}\times\mathbb{R}).
            		\]

            		For a function $\Omega$ which is suitably smooth and bounded (for the purposes of this section say $\Omega\in \mathscr{C}^2_b(S^{n-1}\times\mathbb{R})$ with $	\left\|\Omega\right\|_{\mathscr{C}^{2}_b}\leq c/3$) we can conveniently describe a \textbf{normal graph} over $\textbf{Cat}_c$ by means of the local parametrizations defined by the map
            		\be\label{eq:norgen}
            		F_c^{\Omega}=\left(\left(c \phi(s)+\Omega(\theta,s)\phi^{1-n}(s
            		)\right)\theta,c\psi(s)-\frac{\dot{\phi}(s)}{\phi(s)}\Omega(\theta,s) \right).
            		\ee
            		
            			We shall also work in exponentially weighted Sobolev spaces, namely for $n\geq 2$ (in fact, for $n\geq 3$) we will consider the spaces 
            			\[
            			\mathscr{H}^{k}_{q}(S^{n-1}\times\mathbb{R}):=\left\{f: S^{n-1}\times\mathbb{R}\to\mathbb{R} | \sum_{|\alpha|\leq k}\left(\int_{S^{n-1}\times\mathbb{R}}|\nabla^{\alpha}f(\theta,s)|^2|\phi^{-2q}(s)|\,d\theta ds\right)^{1/2}<\infty \right\}
            			\] endowed with their natural Hilbert norms, which we shall denote by $\|\cdot\|_{\mathscr{H}^k_q}$. %In particular, we stress that these are functional spaces over the (Riemannian) product manifold $S^{n-1}\times\mathbb{R}$ (so that there is no dependence on the neck parameter $c$).

            	\subsection{Notations}\label{subs:nota}
            	
            		We will denote with  $\mathcal{L}^{(k)}: \mathscr{C}^k_b(S^{n-1}\times \R) \to \R$ a  linear functional such that  $\mathcal{L}^{(k)}(\Omega)$ is a linear combination (with bounded coefficients) of $\partial^{\alpha} \Omega$, for $|\alpha|\leq k$. As a result, one has
            		\begin{equation}\label{est:Lk}
            		\| \mathcal{L}^{(k)}(\Omega)\|_{\mathscr{C}^0_b} \leq C  \| \Omega \|_{ \mathscr{C}^k_b} \quad \forall \ \Omega \in\mathscr{C}^k_b (S^{n-1}\times \R)  
            		\end{equation}
            		where $C$ is a constant depending on the coefficients of $\mathcal{L}^{(k)}$ but not on $\Omega$. \newline
            		For higher order quantities we will use the notation $\mathcal{Q}^{(k)}:  \mathscr{C}^k_b (S^{n-1}\times \R) \to \R$ to denote a finite linear combination (with bounded coefficients) of monomials of the form $\prod_{i\in I} \partial^{\alpha_i} \Omega$, for $0\leq|\alpha_i|\leq k, |I|\geq 2$, thus satisfying the estimate
            		\begin{equation}\label{est:Qk}
            		\| \mathcal{Q}^{(k)}(\Omega)\|_{\mathscr{C}^0_b} \leq C  \| \Omega \|^2_ {\mathscr{C}^k_b} \quad \forall \ \Omega \in\mathscr{C}^k_b(S^{n-1}\times \R),  \; \| \Omega \|_{\mathscr{C}^k_b}\leq 1
            		\end{equation}
            		where $C$ is a constant depending on the coefficients of $\mathcal{Q}^{(k)}$ but not on $\Omega$.
            		The same notations will be used for  vectorial and matrix-valued quantities depending on $\partial^\alpha \Omega$ and satisfying the estimates \eqref{est:Lk},  \eqref{est:Qk}. \\
            		\indent Throughout the paper, we let $C$ denote a positive constant which is allowed to vary from line to line and in fact even within the same line. Whenever appropriate, we shall stress the functional dependence of $C$ on some of our parameters by adding subscripts, as per $C_{m}$.
            		We shall agree that constants \textsl{without subscripts} are allowed to depend on the ambient dimension and (for what concerns the proof of Theorem \ref{thm:main} and its ancillary results) to be bounded in terms of $\sup_{m\in [0,\varepsilon_0]}\|\Omega_m\|_{\mathscr{C}^2_b}$ (which is of course finite, by the continuity assumption on the curve $\Omega_{(\cdot)}$ in question). Lastly, we shall adopt the notation $O(t)$ to denote some quantity whose absolute value is bounded by $Ct$ where $C>0$ is a constant without subscripts, in the sense above. Since Theorem \ref{thm:main} concerns \textsl{large} values of the neck and \textsl{small} values of the ambient mass, we convene to tacitly assume a positive lower bound for the neck parameter (say $c\geq \overline{c}\geq 1$) and an upper bound for the mass parameter (say $m\leq \underline{m}\leq 1$) so that, for instance, we shall simply write $O(mc^{-1})$ in lieu of $O(mc^{-1})+O(m^2c^{-3})$, which will allow to group up effectively the error terms we need to deal with.

            	\section{A perturbative non-existence result}
            	\label{sec:nonexist}
            	
            	\subsection{Preliminaries on catenoidal graphs}
            	
            	 		When $n=2$, we will describe the Euclidean catenoid of neck radius $c>0$ and vertical axis, $\textbf{Cat}_c$ in $\mathbb{R}^3$, by means of the obvious restrictions of the smooth covering map $F_c:S^{1}\times \mathbb{R}\to\mathbb{R}^{3}$ given by
            	 		\begin{equation}\label{eq:catc}
            	 		F_c(u,v):=(c\cosh(v)\cos(u), c\cosh(v)\sin(u),cv).
            	 		\end{equation}
            	 		Associated to such parametrizations are the tangent vectors
            	 		\[
            	 		\frac{\partial F_c}{\partial u}=(-c\cosh(v)\sin(u),c\cosh(v)\cos(u),0), \ \
            	 		\frac{\partial F_c}{\partial v}=(c\sinh(v)\cos(u),c\sinh(v)\sin(u),c)
            	 		\]
            	 		whose unit (normalized) counterparts are
            	 		\[
            	 		(\textbf{e}_{1})_{F_c}^{\de}=(-\sin(u),\cos(u),0), \ \ (\textbf{e}_{2})_{F_c}^{\de}=(\tanh(v)\cos(u),\tanh(v)\sin(u),\cosh^{-1}(v))
            	 		\]
            	 		and hence there is an associated unit normal vector given by
            	 		\begin{equation}\label{eq:defn}
            	 		\textbf{n}_{F_c}^{\de}=\cosh^{-1}(v)\left(\cos(u),\sin(u), -\sinh(v)\right).
            	 		\end{equation} 
            	 		Lastly, the general formula for parametrized normal graphs \eqref{eq:norgen} specifies (for $n=2$) to
            	 		\begin{equation}\label{eq:norgraph}
            	 		F_c^\Omega(u,v)=\left(\left(c\cosh(v)+\frac{\Omega(u,v)}{\cosh(v)}\right)\cos(u), \left(c\cosh(v)+\frac{\Omega(u,v)}{\cosh(v)}\right)\sin(u),cv-\tanh(v)\Omega(u,v)\right).
            	 		\end{equation}
            	 		
            	 		In the proof of Theorem \ref{thm:main} we shall need the following two lemmata.
            	 		
            	 		\begin{lemma}\label{lem:FFhat}
            	 			Let $F_c:S^1\times \R \to \R$ (resp.  $F^\Omega_c$)  be the parametrization \eqref{eq:catc} of the catenoid in $\R^3$ with neck $c$ (resp. of the normal graph \eqref{eq:norgraph}), regarded as a submanifold of the flat Euclidean space. Then for every $\Omega\in\mathscr{C}^2_b(S^1\times \R)$, one has
            	 			\begin{align}
            	 			F^\Omega_c &= F_c+ \cL^{(0)} (\Omega) \label{eq:hFF}, \\
            	 			\textbf{n}_{F^\Omega_c}^{\de} &= \textbf{n}_{F_c}^{\de} + \frac{1}{c\cosh(v)}\cL^{(1)}(\Omega) + \frac{1}{c^2\cosh^2(v)}\cQ^{(1)}(\Omega),   \label{eq:hnn} \\
            	 			H_{F^\Omega_c}^{\de} &=  L_{F_c}^{\de} (\Omega) + \frac{1}{c^3 \cosh^3(v)} \cQ^{(2)}(\Omega)  \label{eq:hHH}, \\
            	 			|A_{F_c^\Omega}^{\de}|^{2} &= \frac{2}{c^{2} \, \cosh^{4}(v)} +  \frac{1}{c^3 \cosh^4(v)} \cL^{(2)}(\Omega) + \frac{1 }{c^4 \cosh^{4}(v)} \cQ^{(2)}(\Omega) \label{eq:hA2A2}, 
            	 			\end{align}
            	 			where $\textbf{n}_{F_c}^{\de}$ (resp.  $ \textbf{n}_{F_c^\Omega}^{\de}$) is the normal vector \eqref{eq:defn} to the unperturbed catenoid $F_c$ (resp. to the perturbed catenoid $F_c^\Omega$), $H_{F^\Omega_c}^{\de}$ (resp. $A_{F_c^\Omega}^{\de}$) is the mean curvature (resp. the second fundamental form) of $F_c^\Omega$, and $L_{F_c}^{\de}$ is the Jacobi operator of the catenoid $F_c$ namely
            	 			\begin{equation}\label{eq:defLFc}
            	 			L_{F_c}^{\de} (\Omega):= \frac{1}{c^2 \cosh^2(v)} \left(\frac{\partial ^2 \Omega} { \partial u^2} +  \frac{\partial ^2 \Omega} { \partial v^2} +\frac{2 \Omega} { \cosh^2(v)} \right).
            	 			\end{equation}
            	 		\end{lemma}
            	 		
            	 		\begin{proof}
            	 			From the expression \eqref{eq:norgraph} of the normal graph $F^\Omega_c$ it is clear that
            	 			\[F^\Omega_c = F_c+ \cL^{(0)} (\Omega), \quad \partial_i F^\Omega_c = \partial_i F_c+ \cL^{(1)} (\Omega),  \quad \partial^2_{ij}F^\Omega_c = \partial^2_{ij} F_c+ \cL^{(2)} (\Omega). 
            	 			\]
            	 			It follows that the first fundamental form $g_{F_c^\Omega}^{\de}$ of  $F^\Omega_c$  reads
            	 			\[
            	 			(g_{F^\Omega_c}^{\de})_{ij}= c^2 \cosh^2(v) \left( \delta_{ij} + \frac{\cL^{(1)}(\Omega)}{c \, \cosh(v)}+  \frac{\cQ^{(1)}(\Omega)}{c^2 \, \cosh^2(v)} \right), 
            	 			\]
            	 			which in turn implies that the inverse matrix of $g_{F^\Omega_c}^{\de}$ is given by
            	 			\[
            	 			(g_{F^\Omega_c}^{\de})^{ij}= \frac{1}{c^2 \cosh^2(v)} \left( \delta^{ij} + \frac{\cL^{(1)}(\Omega)}{c \, \cosh(v)}+  \frac{\cQ^{(1)}(\Omega)}{c^2 \, \cosh^2(v)} \right).
            	 			\]
            	 			Making use, once again, of the expression above for $\partial_i F^{\Omega}_c$ in expanding the wedge product $\partial_u F^{\Omega}_c\wedge \partial_v F^{\Omega}_c$ one finds
            	 			\[
            	 			\textbf{n}_{F^\Omega_c}^{\de} =  \textbf{n}_{F_c}^{\de} + \frac{1}{c\cosh(v)}\cL^{(1)}(\Omega) + \frac{1}{c^2\cosh^2(v)}\cQ^{(1)}(\Omega)   \\ ,  
            	 			\]
            	 			and hence, separating the zeroth, first and higher order terms in $\Omega$ we conclude
            	 			\[
            	 			H_{F^\Omega_c}^{\de}= (g_{F^\Omega_c}^{\de})^{ij} \;  (\partial_{ij}F^\Omega_c \cdot \textbf{n}_{F^\Omega_c}^{\de}) = H_{F_c}^{\de}+  L_{F_c}^{\de} (\Omega) + \frac{\cQ^{(2)}(\Omega)}{c^3 \cosh^3(v)}=  L_{F_c}^{\de} (\Omega) + \frac{\cQ^{(2)}(\Omega)}{c^3 \cosh^3(v)}, 
            	 			\]
            	 			where of course the last equality relies on the fact that $H_{F_c}^{\de}\equiv 0$ as the catenoid $F_c$ is a  minimal surface.
            	 			Combining the above expressions we also get
            	 			\[(A_{F^\Omega_c}^{\de})^{i}_{j}= (g_{F^\Omega_c}^{\de})^{ik}  \;  (\partial_{kj}F^\Omega_c \cdot \textbf{n}_{F^\Omega_c}^{\de})=(A_{F_{c}}^{\de})^{i}_{j}+   \frac{1}{c^2 \cosh^2(v)} \cL^{(2)}(\Omega) + \frac{1}{c^3 \cosh^{3}(v)} \cQ^{(2)}(\Omega).
            	 			\]
            	 			Let us now compute $(A_{F_{c}}^{\de})^{i}_{j}$. 
            	 			Since 
            	 			\begin{align}
            	 			\partial^{2}_{uu} F_{c}(u,v)&= \left( -c \cosh(v) \, \cos(u),  -c \cosh(v) \, \sin(u), 0 \right), \nonumber \\
            	 			\partial^{2}_{uv} F_{c}(u,v)&= \left( -c \sinh(v) \, \sin(u),  c \sinh(v) \, \cos(u), 0 \right), \nonumber \\
            	 			\partial^{2}_{vv} F_{c}(u,v)&= \left( c \cosh(v) \, \cos(u),  c \cosh(v) \, \sin(u), 0 \right),\nonumber
            	 			\end{align}
            	 			from the expression \eqref{eq:defn} of the unit normal $\textbf{n}_{F_c}^{\de}$, we obtain that
            	 			\[
            	 			(A_{F_{c}}^{\de})^{i}_{j}= (g_{F_{c}}^{\de})^{ik}\;  \partial^{2}_{kj} F_{c} \cdot \textbf{n}_{F_c}^{\de}= \frac{1}{c \, \cosh^{2}(v)}  \, \omega^{i}_{j},			
            	 			\]
            	 			where $\omega^{i}_{j}$ is the symplectic matrix $\omega^{1}_{1}=-1$, $\omega^{2}_{2}=1$,  $\omega^{1}_{2}=\omega^{2}_{1}=0$.  This in turn implies  
            	 			\[
            	 			|A_{F_{c}}^{\de}|^{2}=\frac{2}{c^{2} \, \cosh^{4}(v)}.
            	 			\]
            	 			Therefore, we can conclude that
            	 			\begin{align}
            	 			|A_{F^\Omega_c}^{\de}|^{2}&=  (A_{F^\Omega_c}^{\de})^{i}_{j} (A_{F^\Omega_c}^{\de})^{j}_{i}= |A_{F_{c}}^{\de}|^{2}+  \frac{|A_{F_c}^{\de}|}{c^2 \cosh^2(v)} \cL^{(2)}(\Omega) + \frac{|A_{F_{c}}^{\de}| }{c^3 \cosh^{3}(v)} \cQ^{(2)}(\Omega)+  \frac{\cQ^{(2)}(\Omega)}{c^4 \cosh^{4}(v)}   \nonumber \\
            	 			&= \frac{2}{c^{2} \, \cosh^{4}(v)} +  \frac{1}{c^3 \cosh^4(v)} \cL^{(2)}(\Omega) + \frac{1 }{c^4 \cosh^{4}(v)} \cQ^{(2)}(\Omega), \nonumber
            	 			\end{align}
            	 			as desired.
            	 		\end{proof}
            	 		
            	 			\begin{remark}\label{rem:boundedgraph}(Quadratic area growth and finite total curvature of $\mathscr{C}^2$-bounded graphs over a catenoid).
            	 				Let $F_c:S^1\times \R \to \R$  be the parametrization \eqref{eq:catc} of the catenoid in $\R^3$ with neck $c$ and let  $\Omega:S^1\times \R \to \R$ with $\|\Omega\|_{\mathscr{C}^{2}_b}\leq K< \infty$. Then Lemma  \ref{lem:FFhat} implies that there exists $C\in (0,\infty)$ such that
            	 				\begin{align}
            	 				dvol_{g_{F_{c}^{\Omega}}^{\de}}&\leq (1+C)  dvol_{g_{F_{c}}^{\de}}=  (1+C)\, c^{2}\cosh(v)^{2} \;  du\, dv , \nonumber \\
            	 				|A_{F^\Omega_c}^{\de}|^{2} &\leq  (1+C)  |A_{F_{c}}^{\de}|^{2} =  (1+C)  \frac{2}{c^{2} \, \cosh^{4}(v)}.  \nonumber
            	 				\end{align}
            	 				Since the catenoid $\textbf{Cat}_c$ has quadratic area growth and finite total curvature, namely there exists a positive constant $C$ such that $|\textbf{Cat}_c\cap B_{R})|\leq C \, R^{2}$ and $\int_{S^{1}\times \R}  |A_{F_{c}}^{\de}|^{2} dvol_{g_{F_{c}}^{\de}}\leq C$, the above formulas imply that the same is true for the normal graph $F^\Omega_c$.
            	 			\end{remark}
            	
            	\subsection{Pointwise changes of mean curvature}
            	
            	\begin{lemma}\label{lem:meancchange}
            		Let $M=\R^{3}\setminus B_{r_0}$ and for $k\geq 1$ let $e_{(\cdot)}:[0,\varepsilon]\to \mathscr{E}^k$ be a differentiable curve with $e_{0}=0$ and consider for small $m>0$ the Riemannian metric 
            		\be\label{eq:gg0}
            		g^{m,e_{m}}_{ij}(x):=\left(1+\frac{m}{2|x|}\right)^{4}\delta_{ij} + (e_{m}(x))_{ij}.
            		\ee  
            		Then, given a surface $S \subset M$ and called  $H_{S}^{g^{m}}$  (respectively $H_{S}^{g^{m,e_{m}}}$)   the mean curvature  of $S$ in metric $g^{m}:=g^{m,0}$  (resp. in metric $g^{m,e_{m}}$), it holds
            		\be\label{eq:HgHg0}
            		H^{g^{m,e_{m}}}_{S}(x)=H^{g^{m}}_{S}(x)+ |A_{S}^{g^{m}}| \, O(m |x|^{-2})+ O(m |x|^{-3}) \quad \forall \ x \in  S,
            		\ee
            		where   $|A_{S}^{g^{m}}|$ denotes the length of the second fundamental form of $S$ in metric $g^{m}$. 
            	\end{lemma}
            	
            	\begin{proof}
            		Throughout the proof, we shall denote by $g_{S}^{m}$ the Riemannian metric on $S$ induced by restriction of $g^{m}$, and by $g^{m,e_{m}}_{S}$ the corresponding one induced by $g^{m,e_{m}}$.
            		Furthermore, we fix a local moving frame $\left\{\textbf{t}_{1},\textbf{t}_{2}\right\}$ on $S$ which is orthonormal with respect to $g^{m}_{S}$, i.e. $\textbf{t}_{i}(x)\in T_{x}S$ and  $g^{m}_{S}(\textbf{t}_{i},\textbf{t}_{j})=\delta_{ij}$.  Clearly, by virtue of the differentiability assumption on $e_{(\cdot)}$ (hence on the resulting pointwise bounds) we have
            		\be\label{eq:gSgS-1}
            		(g^{m,e_{m}}_{S})_{ij}:= g^{m,e_{m}}_{S}(\textbf{t}_{i},\textbf{t}_{j})= \delta_{ij}+O(m|x|^{-2}), \quad (g^{m,e_{m}}_{S})^{ij}:= (g^{m,e_{m}}_{S})^{-1}_{ij}= \delta^{ij}+O(m|x|^{-2}).
            		\ee
            		Moreover, said $\textbf{n}^{g^{m,e_{m}}}_{S}$ (resp. $\textbf{n}^{g^{m}}_{S}$) a choice of unit normal vectors to $S$ with respect to $g^{m,e_{m}}$ (resp. $g^{m}$),  using the relations 
            		\[
            		g^{m,e_{m}}(\textbf{n}^{g^{m,e_{m}}}_{S}, \textbf{t}_{i})=g^{m}(\textbf{n}^{g^{m}}_{S}, \textbf{t}_{i})= 0, \quad i=1,2, \qquad g^{m,e_{m}}(\textbf{n}^{g^{m,e_{m}}}_{S}, \textbf{n}^{g^{m,e_{m}}}_{S})=g^{m}(\textbf{n}^{g^{m}}_{S}, \textbf{n}^{g^{m}}_{S})= 1
            		\]
            		together with \eqref{eq:gSgS-1} and decomposing $\textbf{n}^{g^{m,e_{m}}}_{S}$ with respect to the frame $\left\{\textbf{t}_1,\textbf{t}_2,\textbf{n}_S^{g^{m}}\right\}$, it is easy to check that
            		\be\label{eq:nn0}
            		\textbf{n}^{g^{m,e_{m}}}_{S}=\textbf{n}^{g^{m}}_{S}+O(m|x|^{-2}).
            		\ee
            		If we let $\nabla^{g^{m,e_{m}}}$ and $\nabla^{g^{m}}$ denote the covariant derivative in $(M, g^{m,e_{m}})$ and $(M, g^{m})$ respectively,
            		the decay assumptions on the term $e_{m}$ imply
            		\be\label{eq:nablagnablag0}
            		\nabla^{g^{m,e_{m}}}_{\textbf{t}_{i}} \textbf{t}_{j}=\nabla^{g^{m}}_{\textbf{t}_{i}}\textbf{t}_{j}+ O(m|x|^{-3}).
            		\ee
            		Combining \eqref{eq:gSgS-1},  \eqref{eq:nn0}  and \eqref{eq:nablagnablag0} we obtain
            		\begin{align}
            		H^{g^{m,e_{m}}}_{S}&=-  (g_{S}^{m,e_{m}})^{ij} \, g^{m,e_{m}}( \nabla^{g^{m,e_{m}}}_{\textbf{t}_{i}} \textbf{t}_{j},  \textbf{n}^{g^{m,e_{m}}}_{S}) \nonumber \\
            		&=- (\delta^{ij}+O(m|x|^{-2})) \cdot \nonumber \\
            		&\qquad\cdot(\delta+O(m|x|^{-2})) \Big(\nabla^{g^{m}}_{\textbf{t}_{i}} \textbf{t}_{j}+O(m|x|^{-3}) ,\textbf{n}^{g^{m}}_{S}+O(m|x|^{-2})\Big)  \nonumber \\
            		&= H^{g^{m}}_{S}+ |A_{S}^{g^{m}}| \, O(m|x|^{-2})+ O(m|x|^{-3}),  \nonumber 
            		\end{align}
            		as desired.
            	\end{proof}
            	
            		We proceed by recalling how the mean curvature of a surface changes under a (pointwise) conformal transformation of the ambient metric. 
            		
            		\begin{lemma}\label{lem:confchange}
            			Let $(M^3,g)$ be a Riemannian manifold and let $g_2=f^4g_1$ be a pointwise conformal metric defined by a differentiable, positive factor. Given a surface $S\subset M$ and denoted by $H^{g_1}_S$ (resp. $H^{g_2}_S$) the mean curvature of $S$ in $(M^3, g_1)$ with respect to the unit normal vector $\textbf{n}^{g_1}_S$ (resp. in $(M^3, g_2)$, with respect to $\textbf{n}^{g_2}_S=f^{-2}\textbf{n}^{g_1}_S$) then
            			\[
            			H^{g_2}_S= \frac{1}{f^2}(H^{g_1}_S+4\textbf{n}^{g_1}_S\cdot \partial \log f).
            			\]  	
            		\end{lemma}	
            		
            		\begin{remark}\label{rem:fbounds}
            		Throughout this section, we shall exploit the previous formula with $M^3=\mathbb{R}^3\setminus B_{r_0}$, $g_1$ the flat Euclidean metric and $f(x)=1+\frac{m}{2|x|}$. In this respect, we will repeatedly make use of the following simple estimates: given that
            		\be\label{eq:fomegaesp}
            		|F^\Omega_c|=|F_c|\left(1+\frac{2c\psi_0\Omega}{|F_c|^2}+\frac{|\Omega|^2}{|F_c|^2} \right)^{1/2} \ \textrm{for} \ \psi_0(u,v)=1-v\tanh(v)
            		\ee
            		which implies (for $\Omega\in\mathscr{C}^0_b(S^1\times\mathbb{R})$) that $|F^\Omega_c|=|F_c|(1+O(c^{-1}\cosh^{-1}(v)))$, we have the bounds
            		\be\label{eq:fomegabound}
            		1\leq f(F_c(u,v))\leq 1+ O(mc^{-1}), \ \  f(F^\Omega_c(u,v))=1+O(mc^{-1}).
            		\ee
            		\end{remark}

            	\begin{remark}\label{rem:len2ndff}
            		In order to employ the estimate \eqref{eq:HgHg0}, it is useful to get a bound on   $|A_{F^{\Omega}_{c}}^{g^{m}}|$. To this aim, using the pointwise conformal invariance of $|\mathring{A}|^{2} \, dvol$, we first observe that
            		\begin{align}\label{eq:mathringAgmgmem} \nonumber
            		|\mathring{A}^{g^{m}}_{F_{c}^{\Omega}}|^{2}(u,v)&=\left(1+\frac{m}{2|{F}_{c}^{\Omega}(u,v)|}\right)^{-4} |\mathring{A}^{\delta}_{F_{c}^{\Omega}}|^{2} (u,v) \\ 
            		&\leq \left(1+\frac{m}{2|{F}_{c}^{\Omega}(u,v)|}\right)^{-4} |A^{\delta}_{F_{c}^{\Omega}}|^{2}(u,v)  = O(c^{-2} \cosh^{-4}(v)),
            		\end{align}
            		where in the last estimate we used Remark \ref{rem:boundedgraph} as well as Remark \ref{rem:fbounds}.
            		
            On the other hand
            		\begin{align}\label{eq:meancbound} \nonumber
            		|H^{g^{m}}_{F_{c}^{\Omega}}|^{2}(u,v)&\leq \left|H^{\delta}_{F_{c}^{\Omega}}(u,v)+ 4 \textbf{n}^{\delta}_{F_{c}^{\Omega}}(u,v) \cdot \partial \log\left(1+\frac{m}{2|x|} \right)(F^\Omega_c(u,v)) \right|^{2} \\ 
            		&\leq 2|A^{\delta}_{F_{c}^{\Omega}}|^{2}(u,v)+\frac{8m^2}{|F^\Omega_c(u,v)|^4}  = O(c^{-2} \cosh^{-4}(v)).
            		\end{align}

            		Combining  \eqref{eq:mathringAgmgmem} with \eqref{eq:meancbound}, we infer that
            		\begin{equation}\label{eq:Agmgmem}
            		|A^{g^{m}}_{F_{c}^{\Omega}}|^{2}(u,v)=  |\mathring{A}^{g^{m}}_{F_{c}^{\Omega}}|^{2}(u,v)+  \frac{1}{2}|H^{g^{m}}_{F_{c}^{\Omega}}|^{2}(u,v)     = O(c^{-2} \cosh^{-4}(v)).
            		\end{equation}
            		Plugging \eqref{eq:Agmgmem} into  \eqref{eq:HgHg0} and evaluating along the catenoid, we get
            		\begin{equation}\label{eq:HFcgmgmem}
            		H^{g^{m,e_{m}}}_{F^{\Omega}_{c}}(u,v)=H^{g^{m}}_{F^{\Omega}_{c}}(u,v)+ O(m c^{-3} \cosh^{-3}(v) ).
            		\end{equation}
            			\end{remark}

            			This formula will be used twice in proving Theorem \ref{thm:main}.
            	
            	\subsection{Expansion at infinity for minimal parametrizations}\label{subs:struct}
            	
            	\begin{lemma}\label{lem:expansion}
            		Let $M=\R^{3}\setminus B_{r_0}$ and consider the Riemannian metric
            		\be\label{eq:gg0}
            		g^{m,e}_{ij}(x):=\left(1+\frac{m}{2|x|}\right)^{4}\delta_{ij} + (e(x))_{ij}.
            		\ee  
            		where $e\in\mathscr{E}^2$. Suppose that $F^\Omega_c: S^1\times\mathbb{R}\to (M,g^{m,e})$ is a minimal immersion defined by some $\Omega\in\mathscr{C}^2_b(S^1\times\mathbb{R})$. If $\Omega(u,v)=o(1)$ as $v\to\pm\infty$ then $\Omega\in\mathscr{C}^{2}_{(-1)^{+}}$, namely for all $\varepsilon>0$ one has
            		\[
            		|\Omega(u,v)|+|\nabla \Omega(u,v)|+|\nabla^2\Omega(u,v)|\leq C_{\varepsilon, m, c}|\cosh^{-1+\varepsilon}(v)|.
            		\]
            		
            	\end{lemma}	
            	
            	\begin{remark}\label{rem:comparBR15}
            		The reader may want to compare this result with similar ones that have been obtained for \textsl{Cartesian} minimal graphs in asymptotically flat spaces (see Appendix A in \cite{Car14a} and \cite{BR15}). In particular, Bernard and Rivi\`ere recently proved that (in such setting) a complete minimal surface of finite total curvature and sub-quartic area growth can be decomposed, outside a compact set, in a finite number of graphs with an expansion of the form $p(x')=a\log|x'|+b+q(x')$ where $q(x')=O_2(|x'|^{-1+\varepsilon})$ for all $\varepsilon>0$ (this is Corollary I.2 in \cite{BR15}). The very same conclusion holds true for minimal surfaces of finite Morse index (with no assumption on the area growth), at least if the scalar curvature of the ambient manifold is non-negative (Lemma 14 in \cite{Car14a}). The methods employed in \cite{BR15} are very general and analytically robust, yet significantly less elementary than the ones presented here. 
            	\end{remark}	
            	
            	\begin{proof}
            		
            		Let us start by writing down the partial differential equation solved by $\Omega$. First of all, the minimality assumption, namely the assumption that the mean curvature $H_{F^\Omega_c}^{g^{m,e}}$ vanishes identically, implies (by virtue of Lemma \ref{lem:meancchange} combined with Remark \ref{rem:len2ndff}) the condition $H_{F^\Omega_c}^{g^{m}}(u,v)=O(c^{-3}\cosh^{-3}(v))$. At this stage, we would like to write $H_{F^\Omega_c}^{g^{m}}(u,v)$ as the sum of $H_{F^\Omega_c}^{\de}(u,v)$ plus error terms depending on the decay rate of the metric $g^{m}$. This is accomplished by means of Lemma \ref{lem:confchange}, combined with the estimates \eqref{eq:fomegabound} (see Remark \ref{rem:fbounds}). Let us discuss this computation in detail: keeping in mind the equation for the conformal change of mean curvature, it is enough to combine the identity
            			\begin{align}\label{eq:splitdiff} 
            			\textbf{n}_{F_c^{\Omega}}^{\de}\cdot\partial\log f(F_c^{\Omega}(u,v))&= \textbf{n}_{F_c}^{\delta} \cdot \partial \log f(F_c(u,v))+(\textbf{n}_{F_c^{\Omega}}^ {\delta}-\textbf{n}_{F_c}^{\de}) \cdot \partial \log f(F_c (u,v)) \nonumber \\
            			&+ \textbf{n}_{F_c^\Omega}^\delta \cdot  (\partial \log f(F_c^{\Omega}(u,v))-\partial \log f( F_c(u,v))) 
            			\end{align}
            			with
            			\begin{itemize}
            			 \item{the equation $	\textbf{n}_{F^\Omega_c}^{\de}=\textbf{n}_{F_c}^{\de} + \frac{1}{c\cosh(v)}\cL^{(1)}(\Omega) + \frac{1}{c^2\cosh^2(v)}\cQ^{(1)}(\Omega)$ (derived in Lemma \ref{lem:FFhat});}
            			 \item{the estimate
            			\[
            			|\partial \log f(F_c^{\Omega}(u,v))-\partial \log f( F_c(u,v))|\leq\frac{\mathcal{L}^{(0)}(\Omega)}{|F_c(u,v)|^3}+\frac{\mathcal{Q}^{(0)}(\Omega)}{|F_c(u,v)|^4}; 
            			\]
            		}
            			\item{the explicit calculation of the first (and leading) term: one has $\partial \log f=\frac{\partial f}{f}$ and
            	$\partial f(x)=-\frac{m}{2} \frac{x}{|x|^3}$
            			which, evaluated on the unperturbed catenoid $F_c$ defined in \eqref{eq:catc} gives
            			\begin{equation}\label{eq:partialfFc}
            			\partial f(F_c(u,v))=-\frac{m}{2c^2} \frac{(\cosh(v) \cos(u), \cosh(v) \sin(u), v)}{|\cosh^2(v)+v^2|^{3/2}}
            			\end{equation}
            			and hence (recalling the explicit expression \eqref{eq:defn} of $\textbf{n}_{F_c}^{\delta}$)
            			\begin{equation}\label{eq:NdotpartialfFc}
            			\textbf{n}_{F_c}^{\delta} \cdot \partial f(F_c(u,v))=-\frac{m}{2c^2} \frac{1-v \tanh (v)} {|\cosh^2(v)+v^2|^{3/2}},
            			\end{equation}
            		}
            			\end{itemize}
            			to finally get 
            					 \be\label{eq:intermedH}
            			 H_{F^\Omega_c}^{g^{m}}(u,v)=H_{F^\Omega_c}^{\de}(u,v)-\frac{2m}{c^2}\frac{1-v\tanh(v)}{|\cosh^2(v)+v^2|^{3/2}}+O(c^{-3}\cosh^{-3}(v)).
            			 \ee
            			 At this stage, as we had explained at the beginning of this proof the minimality assumption implies that in fact  
            					 \[
            					 H_{F^\Omega_c}^{\de}(u,v)=\frac{2m}{c^2}\frac{1-v\tanh(v)}{|\cosh^2(v)+v^2|^{3/2}}+O(c^{-3}\cosh^{-3}(v)).
            					 \]
            			and thus if we plug-in \eqref{eq:hHH} and \eqref{eq:defLFc} and multiply by $c^2\cosh^2(v)$  we derive for $\Omega$ the equation
            			\[
            			\frac{\partial^2 \Omega}{\partial u^2}+\frac{\partial^2\Omega}{\partial v^2}= \Gamma (u, v), \ \textrm{where} \ \Gamma\in \mathscr{C}^{0}_{(-1)^+}.
            			\]
            			Now, this can be analyzed by separation of variables. We shall outline the argument in order to make our proof self-contained, in spite of the fact that somewhat similar discussions are certainly present in the literature (see e.\,g.\ \cite{Mey63}) even though in a slightly different setting and often with less elementary tools. 
            			
            			%Let us observe that, by Liouville's theorem, the equation above can have at most one \textsl{bounded} solution: we will now exhibit (in fact: construct) one that decays exponentially along the ends of $S^1\times\mathbb{R}$ and so such conclusion will be true for $\Omega$.
            			
            			Given any $\eta\in (0,1)$, it follows from the equation above that $\Gamma \in\mathscr{H}^0_{-\eta}(S^1\times\mathbb{R})$ and standard Fourier analysis provides (continuous) functions $\Gamma_0, \Gamma'_j, \Gamma''_j :\mathbb{R}\to\mathbb{R}, (j\geq 1)$ all belonging to $\mathscr{H}^0_{-\eta}(\mathbb{R})$ such that
            			\[
            			\Gamma=\Gamma_0+\sum_{j\geq 1}(\Gamma'_j \Sigma'_j+\Gamma''_j\Sigma''_j) \ \ \textrm{in} \ \mathscr{H}^0_{-\eta}(S^1\times\mathbb{R})
            			\]
            			for $\Sigma'_j=\cos(j u), \Sigma''_j=\sin(j u)$. These functions are defined by integration in the angular variable $u$, namely by means of the well-known formulae
            			\[
            			\Gamma_0(v)=\frac{1}{2\pi}\int_{S^1}\Gamma(u,v)\,du, \ \Gamma'_j(v)=\frac{1}{\pi}\int_{S^1}\Gamma(u,v)\Sigma'_j(u)\,du, \ \Gamma''_j(v)=\frac{1}{\pi}\int_{S^1}\Gamma(u,v)\Sigma''_j(u)\,du,
            			\]
            			and the integral decay (in the sense of the weighted Hilbertian Sobolev spaces above) follows at once by means of Parseval's identity, since such decomposition is patently orthogonal.
            		Furthermore, the assumptions on $\Omega$ (namely the fact that we are postulating $\Omega\in\mathscr{C}^2_b(S^1\times\mathbb{R})$) ensure that its Fourier coefficients $\Omega_0, \Omega'_j, \Omega''_j$ all belong to $\mathscr{C}^2_b(\mathbb{R})$ and, exploiting orthogonality once again in the usual way, solve the family of ODEs given by
            			\be\label{eq:famodes}
            			\frac{d^2\Omega_0}{dv^2} = \Gamma_0, 
            			\frac{d^2\Omega'_j}{dv^2}-j^2\Omega'_j = \Gamma'_j   \ \textrm{and} \
            			\frac{d^2\Omega''_j}{dv^2}-j^2\Omega''_j = \Gamma''_j   \ \textrm{for} \ j\geq 1.
            	\ee
            	These equations can be studied by means of the method of variations of constants (since, obviously, the homogeneous problem can be integrated explicitly). Let us see this argument in detail:
            	\begin{itemize}
            	\item{when $j=0$, the space of solutions of the homogeneous problem is $\langle 1, v\rangle_{\mathbb{R}}$, the Wronskjan determinant is given by $W(v)=1$ and one can write (for some constants $A_0, B_0\in\mathbb{R}$) 
            		\[
            		\Omega_0(v)=v\left[A_0-\int_{0}^{v}\Gamma_0(w)\,dw \right]+1\left[B_0+\int_0^{v}w\Gamma_0(w)\,dw\right].
            		\] 
            		For what concerns the analysis as one lets $v\to+\infty$ one just needs to observe that
            		\[
            		\int_v^{+\infty}|\Gamma_0(w)|\,dw\leq\|\Gamma_0\|_{\mathscr{H}^0_{-\eta}}\left(\int_{v}^{+\infty}\cosh^{-2\eta}(w)\,dw\right)^{1/2}\leq C\|\Gamma_0\|_{\mathscr{H}^0_{-\eta}}\cosh^{-\eta}(v)
            		\]
            		and
            		\begin{align*} 
            		\int_v^{+\infty}|w\Gamma_0(w)|\,dw&\leq \|\Gamma_0\|_{\mathscr{H}^0_{-\eta}}\left(\int_{v}^{+\infty}w^2\cosh^{-2\eta}(w)\,dw\right)^{1/2} \\ \nonumber
            		&\leq C_{\eta'}\|\Gamma_0\|_{\mathscr{H}^0_{-\eta}}\cosh^{-(\eta-\eta')}(v) \ \ \forall \ \eta'\in (0,\eta/4)
            		\end{align*}
            		both gotten by means of the Cauchy-Schwarz inequality, so that in the end
            		\[
            		\Omega_0(v)=\lambda_{+}+\mu_{+}v+O(\cosh^{-(\eta-\eta')}(v)).
            		\]
            		At that stage, we can infer $\lambda_{+}=\mu_{+}=0$ thanks to the fact that $\Omega_0(u,v)=o(1)$ as one lets $v\to+\infty$.
            		An identical analysis can be performed as one lets $v\to-\infty$.
            		 }	
            	\item{when $j \geq 1$, the space of solutions of the homogeneous problem is $\langle e^{jv}, e^{-jv} \rangle_{\mathbb{R}}$, the Wronskjan determinant is given by $W(v)=2j$ and one can write (for some constants $A'_j, B'_j$ and $A''_j, B''_j$)
            		\[
            		\Omega'_j(v)=\frac{e^{jv}}{2j}\left[A'_j-\int_{0}^{v}e^{-jw}\Gamma'_j(w)\,dw\right]+\frac{e^{-jv}}{2j}\left[B'_j+\int_{0}^{v}e^{jw}\Gamma'_j(w)\,dw\right],
            		\]
            			\[
            			\Omega''_j(v)=\frac{e^{jv}}{2j}\left[A''_j-\int_{0}^{v}e^{-jw}\Gamma''_j(w)\,dw\right]+\frac{e^{-jv}}{2j}\left[B''_j+\int_{0}^{v}e^{jw}\Gamma''_j(w)\,dw\right].
            			\]
            			This implies, by virtue of the boundedness of $\Omega'_j$ (resp. $\Omega''_j$), that 
            			\[
            			A'_j=\int_0^{+\infty}e^{-jw}\Gamma'_j(w)\,dw, \ \ B'_j=\int_{-\infty}^{0}e^{jw}\Gamma'_j(w)\,dw,
            			\]
            			(resp. $A''_j=\int_0^{+\infty}e^{-jw}\Gamma''_j(w)\,dw, \ B''_j=\int_{-\infty}^{0}e^{jw}\Gamma''_j(w)\,dw$).

            			At this stage, it follows at once that
            			\[
            			|\Omega'_j(v)|\leq C \frac{\|\Gamma'_j\|_{\mathscr{H}^0_{-\eta}}}{2j}\left(\frac{1}{\sqrt{2(j+\eta)}}+\frac{1}{\sqrt{2(j-\eta)}} \right)\cosh^{-\eta}(v)
            			\]
            			where, of course, $C$ is a positive constant which does not depend on $j$.}	
            	\end{itemize}	
            	Now, we recollect all of the terms above: for $j\geq 1$, we first observe that
            	\[
            	\int_{\mathbb{R}}|\Omega'_j|^2\cosh^{2(\eta-\eta')}(w)\,dw \leq C \frac{\|\Gamma'_j\|^2_{\mathscr{H}^0_{-\eta}}}{j^2(j-\eta)}	\]	(and identically for the terms $\Omega''_j$) so that in the end the series
            	\[
            	\sum_{j\geq 1}\left( \Omega'_j(v)\Sigma'_j(u)+\Omega''_j(v)\Sigma''_j(u)\right)
            	\]	does converge in $\mathscr{H^{0}}^{0}_{-(\eta-\eta')}(S^1\times\mathbb{R})$, and as a result  $\Omega\in \mathscr{H}^{0}_{-(\eta-\eta')}(S^1\times\mathbb{R})$ which holds true for every $\eta\in (0,1)$ and every $\eta'\in (0,\eta/4)$. Hence, elliptic regularity ensures that $\Omega\in \mathscr{H}^{2}_{-(\eta-\eta')}(S^1\times\mathbb{R})$, the Sobolev embedding Theorem (for weighted Sobolev spaces, over asymptotically cylindrical manifolds) provides $\Omega\in \mathscr{C}^{0}_{-(\eta-2\eta')}(S^1\times\mathbb{R})$ and thus the arbitrarity of $\eta'$ and $\eta$ forces $\Omega\in \mathscr{C}^{0}_{(-1)^+}$, hence in fact $\Omega\in \mathscr{C}^{2}_{(-1)^+}$ which completes the proof.
            	
            	\end{proof}	
            	
            	Given the expansion provided by Lemma \ref{lem:expansion} (and the corresponding improved estimates one can derive), we can legitimately perform integration by parts (over $S^1\times\mathbb{R}$) of the term $\psi_0 L_{F^\Omega_c}(\Omega)$ under the sole assumption that $\Omega(u,v)=o(1)$ as $v\to\pm\infty $, where (as above) $\psi_0(u,v)=1-v\tanh(v)$, the Jacobi function associated to the rescaling action on the catenoid, namely $\psi_0=\textbf{n}^{\de}_{F_c}\cdot\frac{\partial F_c}{\partial c}$. As a result, we have the following conclusion.
            	
            	\begin{lemma}\label{lem:IntParts}
            		Let $\psi_0:S^1\times \R \to \R$, $\psi_0(u,v):=1- v \, \tanh(v)$. For every $c>r_0$ and $\Omega \in \mathscr{C}^2_b(S^1\times \mathbb{R})$ defining function of a minimal immersion, such that $\Omega(u,v)=o(1)$ as $v\to\pm\infty$, one has
            		\begin{equation}
            		\int_{S^1\times \R} \psi_0 \; L_{F_c}(\Omega) \; c^2\cosh^2(v) \; du \, dv=0.
            		\end{equation}
            	\end{lemma}

            	\subsection{Completing the proof of the main theorem}\label{subs:endproof}
            	
            	In this subsection, we will exploit all of the ancillary results obtained above to present a direct proof of Theorem \ref{thm:main}.
            	
            	\begin{proof}
            		Let $e_{(\cdot)}:[0,\varepsilon_0] \to \mathscr{E}^{2}(\R^3\setminus B_{r_0})$  be a continuous curve differentiable at $m=0$ with $e_0=0$, and consider   $g^{m,e_m}_{ij}:=\left(1+\frac{m}{2|x|} \right)^4 \delta_{ij} + (e_m(x))_{ij}$ the corresponding  perturbation of the Schwarzschild  metric. 
            		Assume by contradiction that there exists a  continuous curve $\Omega_{(\cdot)}:[0,\varepsilon_0] \to \mathscr{C}^{2}_{b}(S^1\times \R)$, $m\mapsto \Omega_m$   differentiable at $m=0$ with $\Omega_0=0$ such that, for some  $c\geq \bar{c}>1$ and $m\in [0,\varepsilon_0]$, the immersion $F^{\Omega_m}_c$ defined by \eqref{eq:norgraph}  is minimal in $(\R^3\setminus B_1(0), g^{m,e_m})$. We will fix the value of $\overline{c}$ later along the course of the proof.
            		
            		We shall start by multiplying $H^{g^{m,e_{m}}}_{F^{\Omega_{m}}_{c}}\equiv 0$  by  $
            		f^2(F^{\Omega}_c(u,v))\psi_0(v)$ and integrating: thereby we get
            		\begin{align}
            		0&= \int_{S^{1}\times \R}  f^2(F^{\Omega}_c(u,v))H^{g^{m,e_{m}}}_{F^{\Omega_{m}}_{c}} \, \psi_{0} \, dvol_{F_c}^{\delta} \nonumber \\
            		&=    \int_{S^{1}\times \R} f^2(F^{\Omega}_c(u,v)) H^{g^{m}}_{F^{\Omega_{m}}_{c}} \, \psi_{0} \, dvol_{F_c}^{\delta} +   \int_{S^{1}\times \R}   O(m c^{-3} \cosh^{-3}(v))  \, \psi_{0} \, dvol_{F_c}^{\delta}   \nonumber \\
            		&=    \int_{S^{1}\times \R}  f^2(F^{\Omega}_c(u,v))H^{g^{m}}_{F^{\Omega_{m}}_{c}} \, \psi_{0} \, dvol_{F_c}^{\delta} + O(m c^{-1})    \label{eq:mc-alpha}
            		\end{align}
            		where we have used (for what concerns the second summand) the preliminary estimates contained in Remark \ref{rem:len2ndff}.
            		At this stage, specializing Lemma \ref{lem:confchange} to  the case when $f=1+\frac{m}{2|x|}$ and $S=F_c^{\Omega_m}(S^1\times\mathbb{R})$ we get
            		 \begin{equation}\label{eq:contrad}
            		 \int_{S^1 \times \R} H^{\de}_{F^{\Omega_{m}}_{c}} \, \psi_0 \, dvol_{F_c}^{\delta}=O(mc^{-1})  -4 \int_{S^1\times \R} \psi_0 \,  \textbf{n}_{F_c^{\Omega_m}}^{\delta}(u,v) \cdot \partial\log f (F^{\Omega_m}_c(u,v)) \, dvol_{F_c}^{\delta}.
            		 \end{equation}
            		 We are going to show that for small mass the left-hand side is of  order $m^2$ while the right-hand side is of order $m$; this will give the contradiction, for $m>0$ small enough, and will prove Theorem \ref{thm:main}. To this aim we compute separately the two sides of \eqref{eq:contrad}, starting with the left one.
            		 
            		 Since by assumption $[0,\varepsilon_0]\ni m\to \Omega_m \in \mathscr{C}^2_b(S^1\times \R)$ is a continuous curve differentiable at $m=0$ with $\Omega_0=0$, up to choosing a smaller $\varepsilon_0>0$, we have 
            		 \begin{equation}\label{eq:OmegaCm}
            		 \| \Omega_m \|_{\mathscr{C}^2_b} \leq C\, m \quad \forall \ m\in [0,\varepsilon_0], 
            		 \end{equation}
            		 for some $C>0$ depending on the curve $\Omega_{(\cdot)}$ but not on $m$. Therefore, the very definition \eqref{est:Qk} of  $\cQ^{(2)}$  gives
            		 \begin{equation}\label{eq:estQ2Om}
            		 \| \cQ^{(2)}(\Omega_m) \|_{\mathscr{C}^{0}_{b}} \leq C \| \Omega_m \|_{\mathscr{C}^2_b}^2 \leq   C m^2 \quad \forall \ m\in [0,\varepsilon_0].
            		 \end{equation}
            		 Putting together Lemma \ref{lem:FFhat} (equation \eqref{eq:hHH})
            		  with Lemma \ref{lem:IntParts} (whose applicability relies on the assumption that, for each fixed $m$, one has $\Omega_m(u,v)=o(1)$ for $v\to\pm\infty$),  we infer that for every  $m\in [0,\varepsilon_0]$ 
            		 \begin{align}
            		 \left| \int_{S^1 \times \R}  H_{F_c^{\Omega_m}}^{\delta} \, \psi_0 \, dvol_{F_c}^{\delta} \right| &\leq  \left| \int_{S^1 \times \R}   L^{\de}_{F_c}(\Omega_m) \, \psi_0 \, dvol_{F_c}^{\delta} \right|  + \left|  \int_{S^1 \times \R}  \frac{\psi_0(v)}{c \cosh(v)}\cQ^{(2)}(\Omega_m) \, du dv \right| \nonumber \\
            		 &=   \left|  \int_{S^1 \times \R}  \frac{\psi_0(v)}{c \cosh(v)}\cQ^{(2)}(\Omega_m) \, du dv \right|   \leq C \frac{m^2}{c}, \label{eq:EstLHS}
            		 \end{align}
            		 where in the last estimate we used  \eqref{eq:estQ2Om}.
            		 
            		 	We now compute the right-hand side of \eqref{eq:contrad}. First of all let us rewrite it isolating the main term and the remainders:
            		 	\begin{align}
            		 	\int_{S^1\times \R} \psi_0 \,  \textbf{n}_{F_c^{\Omega_m}}^{\delta} &\cdot \partial \log f(F_c^{\Omega_m}(u,v)) \, dvol_{F_c}^{\delta}=\int_{S^1\times \R} \psi_0 \,  \textbf{n}_{F_c}^{ \delta} \cdot \partial \log f( F_c(u,v)) \, dvol_{F_c}^{\delta} \nonumber \\
            		 	&+ \int_{S^1\times \R} \psi_0 \,  ( \textbf{n}_{F_c^{\Omega_m}}^{\delta}-\textbf{n}_{F_c}^{\delta}) \cdot \partial \log f(F_c (u,v)) \, dvol_{F_c}^{\delta} \nonumber \\
            		 	&+ \int_{S^1\times \R} \psi_0 \,  \textbf{n}_{F_c^{\Omega_m}}^{\delta} \cdot  (\partial \log f(F_c^{\Omega_m}(u,v))-\partial \log f( F_c(u,v)))  \, dvol_{F_c}^{\delta}. \label{eq:RHS1} 
            		 	\end{align}
            		 	We claim that the first integral is the main term (of order $m$) and the last two lines are of order $m^2$. Recalling \eqref{eq:NdotpartialfFc}, the first integral in \eqref{eq:RHS1} can be written more explicitly as
            		 	\be\label{eq:1integral0}
            		 	\int_{S^1\times \R} \psi_0 \,  \textbf{n}_{F_c}^{\delta} \cdot \partial \log f( F_c(u,v)) \, dvol_{F_c}^{\delta} =- \frac{m}{2} \int_{S^1\times \R}     \frac{(1-v \tanh(v))^2 \; \cosh^2(v)}   {|\cosh^2(v)+v^2|^{3/2}  f(F_c(u,v)) } \; du dv.
            		 	\ee
            		 	
            		 	As a result, we get
            		 	\begin{equation}\label{eq:1integral}
            		 	\left| \int_{S^1\times \R} \psi_0 \,  \textbf{n}_{F_c}^{\delta} \cdot \partial \log f( F_c(u,v)) \, dvol_{F_c}^{\delta} + A \, m\right| \leq C\, \frac{m^2}{c}, 
            		 	\end{equation}
            		 	where we have set 
            		 	\begin{equation}\label{eq:defA}
            		 	A:= \frac{1}{2} \int_{S^1\times \R}     \frac{(1-v \tanh(v))^2 \; \cosh^2(v)}   {|\cosh^2(v)+v^2|^{3/2}  } \; du dv > 0.
            		 	\end{equation}
            		 	
            		 		With analogous estimates, using Lemma \ref{lem:FFhat}, it is not difficult to see that
            		 		\begin{align}
            		 		&\left| \int_{S^1\times \R} \psi_0 \,  ( \textbf{n}_{F_{c}^{\Omega_m}}^{\delta}-\textbf{n}_{F_c}^ {\delta}) \cdot \partial \log f(F_c (u,v)) \, dvol_{F_c}^{\delta} \right| \nonumber \\
            		 		&\quad +\left| \int_{S^1\times \R} \psi_0 \,  \textbf{n}_{{F}_c^{\Omega_m}}^{\delta} \cdot  (\partial \log f(F_c^{\Omega_m}(u,v))-\partial \log f(F_c(u,v)))  \, dvol_{F_c}^{\delta}\right| \leq C \, m^2 \label{eq:2integral} 
            		 		\end{align}
            		 		the bound for the second summand relying on the fact that
            		 		\[
            		 		|\partial \log f(F_c^{\Omega_m}(u,v))-\partial \log f( F_c(u,v))|\leq m\left(\frac{\mathcal{L}^{(0)}(\Omega_m)}{|F_c(u,v)|^3}+\frac{\mathcal{Q}^{(0)}(\Omega_m)}{|F_c(u,v)|^4}\right). 
            		 		\]
            		 	Plugging  \eqref{eq:EstLHS}, \eqref{eq:RHS1},     \eqref{eq:1integral} and  \eqref{eq:2integral} into \eqref{eq:contrad} implies that if $F^{\Omega_m}_{c_m}$ is a minimal immersion into $(\R^3\setminus B_1, g^{m,e_m})$ then 
            		 	\begin{equation}\nonumber
            		 	A m \leq  C \left( m^2 + \frac{m}{c} \right),
            		 	\end{equation}
            		 	for some constant $C>0$ depending on the curves $\Omega_{(\cdot)},e_{(\cdot)}$ but not on $m$. Thus, we can certainly choose $\bar{c}$ large enough that
            		 	\[
            		 	\frac{C}{c}\leq\frac{A}{2} \ \ \ \forall \ c\geq \bar{c}.
            		 	\] 
            		 	Hence we should have for our sequence $\left\{m_i\right\}$ of positive numbers converging to zero
            		 	\[
            		 	\frac{A}{2}m_i \leq C m^2_i 
            		 	\]
            		 	which is definitely not possible when $i$ is large enough. This contradiction completes the proof.

            		            	\end{proof}	
            		            	
            		            	%We shall conclude this section with a few remarks concerning the proof we have presented, and possible extensions thereof.
            		            	
            		            	\begin{remark}
            		            	As should be clear to the reader, the argument we have described above is of \textsl{global} nature since it relies on integration by parts and on the fact that the function $\psi_0$ defines a Jacobi field on the catenoid. It is then natural to wonder whether a modification (or, rather, an improvement) of such argument, local with respect to a given end (of the surface), can be produced in order to rule out the existence of catenoidal ends for, say, complete minimal surfaces of finite Morse index or finite total curvature (cmp. Remark \ref{rem:comparBR15}). That does not seem to be the case. In fact, following this strategy one can only rule out the existence of catenoidal-type ends with \textsl{Euclidean} expansion\footnote{It is well-known \cite{Sch83} that a minimal end of finite total curvature in (flat) $\mathbb{R}^3$ has an expansion of the type above, as can be seen either via a suitable Weierstrass parametrization or, directly, from the minimal surface equation in Euclidean coordinates. Such conclusion is simply not true for general asymptotically Schwarzschildean metrics, see \cite{Mey63} for a broad discussion of the asymptotic expansion of solutions to elliptic equations.}, namely of every end that can be written (in suitable asymptotically flat coordinates $\left\{x\right\}$) as an outer-graph of the form
            		            	\be\label{eqref:euclexp}
            		            	p(x')=a\log|x'|+b+O(|x'|^{-1}) \ \textrm{where} \ x'=(x^1,x^2)
            		            	\ee
            		            	for some $a\neq 0$. For indeed, if that were the case, a rather elementary argument would allow (possibly changing the asymptotically flat coordinates) to parametrize the end in question by means of a map $F^{\Omega}_c: S^1\times [0,\infty)\to\mathbb{R}^3$ where $\Omega\in\mathscr{C}^{k}_{-1}$ and hence the condition that $H_{F_c^{\Omega}}^{g^{m,e_m}}\equiv 0$ would imply

            		            		\[
            		            		\frac{1}{c^2 \cosh^2(v)} \left(\frac{\partial ^2 \Omega} { \partial u^2} +  \frac{\partial ^2 \Omega} { \partial v^2} +\frac{2 \Omega} { \cosh^2(v)} \right)+O(\cosh^{-5}(v))=\frac{2m}{c^2}\frac{1-v\tanh(v)}{(\cosh^2(v)+v^2)^{3/2}}+O(\cosh^{-3}(v))
            		            		\]
            		            		thus, due to the decay assumptions on $\Omega(u,v)$
            		            		\[
            		            		\frac{2m}{c^2}\frac{1-v\tanh(v)}{(\cosh^2(v)+v^2)^{3/2}}=O(\cosh^{-3}(v))
            		            		\]
            		            		so that, multiplying by $\cosh^3(v)$ and rearranging terms (based on the assumption that $m>0$)
            		            		\[
            		            		v=O(1) \ \ \textrm{for} \ |v|\to\infty
            		            		\]	
            		            		which is a patent contradiction. Yet, this argument does not go through in case one only knows that the parametrization holds with $\Omega\in \mathscr{C}^{k}_{(-1)^{+}}$ which is all one can expect for the general class of metrics we are dealing with.
            		            		            	\end{remark}

            	\section{Higher-dimensional minimal catenoids in asymptotically Schwarzschildean spaces}\label{sec:app}
            	
            The goal of this section is to show that the conclusion of Theorem \ref{thm:main} is strictly three-dimensional, as in higher dimensions one can indeed perturb Euclidean catenoids of large neck in order to obtain catenoidal minimal hypersurfaces in asymptotically Schwarzschildean spaces.
            
            As mentioned in the introduction, for the sake of brevity we shall restrict our discussion to radially symmetric (asymptotically flat) metrics on $\mathbb{R}^{n+1}\setminus B_{r_0}$. More precisely, given $k\geq 2$ we denote with $\mathscr{M}^{k}$ the closure of the space of $\mathscr{C}^k$ maps from $[r_0,\infty)$ into $(n+1)\times (n+1)$ real symmetric matrices such that the following norm
            \be\label{eq:normg} 
            \|h\|_{\mathscr{M}^{k}}:= \sum_{i,j=1}^{n+1}\ \sum_{0\leq a\leq k} \ \sup_{r\geq r_0} \left|\frac{d^a h_{ij}}{dr^a}\right| r^{n-1+a} 
           \ee
            is finite.

             In order to apply an Implicit Function Theorem type argument, one has to understand the kernel of $L_{F_c}^{\de}$. It is readily seen that the only Jacobi fields\footnote{Let us agree to describe each Jacobi field (say $\textbf{v}$) by means of the associated function $J=\textbf{n}^{\de}_{F_c}\cdot\textbf{v}$ that belongs to the kernel of $L_{F_c}^{\de}$.} on $\textbf{Cat}_c$ \textsl{not depending on the angular variable $\theta$} are:

            \begin{itemize}
            	\item the function
            	\begin{equation}\label{eq:defJtrans}
            	S^1_0(s) := \frac{\dot{\phi}(s)}{\phi(s)}.
            	\end{equation}
            	which is associated to the translation in the $\mathbb{R}$ factor in $\mathbb{R}^{n+1}=\mathbb{R}^{n}\times\mathbb{R}$;
            	\item the function
            	\begin{equation}\label{eq:defJ0}
            	S^2_{0}(s):= \frac{\psi(s) \dot{\phi}(s)}{\phi(s)}-\frac{1}{\phi^{n-2}(s)}.  
            	\end{equation}
            	which is associated to dilations (centered at the origin) in $\mathbb{R}^{n+1}$.
            \end{itemize}

            	We shall work with the \textsl{exponentially} weighted functional spaces $\mathscr{C}^{k}_q(S^{n-1}\times\mathbb{R})$  (see Section \ref{sec:setup}) in the very special case $n=1$, in which case those definitions above are meant to describe (via straightforward modifications) the corresponding functional spaces over the real line, namely the spaces  $\mathscr{C}^{k}_q(\mathbb{R})$. We remark that  $\mathscr{C}^{k}_q(\mathbb{R})$ can be canonically identified with the subspace of functions in $\mathscr{C}^{k}_q(S^{n-1}\times\mathbb{R})$ not depending on $\theta$ by means of the trivial map
            	\[
            	\Pi: \mathscr{C}^{k}_q(\mathbb{R}) \to \mathscr{C}^{k}_q(S^{n-1}\times\mathbb{R}), \ \Pi(f)(\theta,s)=f(s)
            	\]
            	 and such identification will always be implicit in the sequel of this section.

            	We consider the operator
            	\[
            	T:=\frac{\partial^2}{\partial s^2}+(n-2)\frac{\dot{\phi}}{\phi}\frac{\partial}{\partial s}+\Delta_{S^{n-1}}+\frac{n(n-1)}{\phi^{2n-2}}
            	\]
            	(which is conjugate to $L^{\de}_{F_c}$ since in fact $L_{F_c}^{\de}=c^{-2}\phi^{-2}T$ but has the virtue of being independent of $c$).

            		If $\tilde{f}:S^{n-1}\times\mathbb{R}$ is a function of the sole variable $s$, then patently $T\tilde{f}=T_0\Pi^{-1}(\tilde{f})$ where
            		
            	\[
            		T_0:=\frac{\partial^2}{\partial s^2}+(n-2)\frac{\dot{\phi}}{\phi}\frac{\partial}{\partial s}+\frac{n(n-1)}{\phi^{2n-2}}.
            		\]

            		  We consider for $q\in (-(n-2),0)$  the mean curvature map
            		    \begin{equation}\label{eq:defmapH1}
            		    H: \mathscr{M}^{(k-1)\vee 2}\times \mathscr{C}^{k}_q(\R)\to \mathscr{C}^{k-2}_{q-2} (\R)
            		    \end{equation}
            		    given by
            		    \begin{equation}\label{eq:defmapH2}
            		    H(h, \Omega)= H_{F^{\Omega}_{c}}^{\de+h},
            		    \end{equation}
            		    that is to say the mean curvature of (the trace of) $F^{\Omega}_{c}$ in metric $g=\de+h$.  Making use of the higher dimensional analogues of Lemma \ref{lem:meancchange} and Lemma \ref{lem:FFhat}, the map in question is routinely checked to be $C^1$ (in the sense of Calculus in Banach spaces) and of course $H(0,0)=0$ as well as $D_2 H(0,0)=c^{-2}\phi^{-2}(s)T_0$. \\
            		 \indent In this section, we shall outline the proof of the following perturbative existence result:

            				\begin{proposition}\label{pro:highdim}
            					For $n\geq 3$ consider the manifold with boundary $M:=\R^{n+1}\setminus B_{r_0}$, let $k\geq 2$ and let $q\in (-(n-2),0)$. 
            					Then there exists $\varepsilon>0$ and a differentiable map $\mathscr{G}:\mathscr{B}\to \mathscr{C}^{k}_q(\R) $ defined on the ball of radius $\varepsilon$ in $\mathscr{M}^{(k-1)\vee 2}$, with the property that $\mathscr{G}(h)$ is the defining function of a catenoidal minimal hypersurface (over $\textbf{Cat}_{2r_0}$) in $(M,\de+h)$  (that is to say: the map $F^{\mathscr{G}(h)}_{2r_0}$ parametrizes a minimal hypersurface in such ambient manifold). As a result, given any $c\geq 2r_0$ and set $\Phi(x)=\left(\frac{2r_0}{c}\right)x, \ \tilde{h}(y)=h(\Phi^{-1}(y))$, if the inequality
            					\begin{align*}
            					\left(\frac{2r_0}{c}\right)^{n-1} \ \sum_{i,j=1}^{n+1}\ \sum_{0\leq a\leq k} \ \sup_{r\geq c/2}  \left|\frac{d^a h_{ij}}{dr^a}\right| r^{n-1+a} <\varepsilon
            					\end{align*}
            					is satisfied one has that
            					$F^{\mathscr{G}(\tilde{h})c}_c$ parametrizes a minimal hypersurface in $(M,\de+h)$.
            				\end{proposition}	
            				
            				We remark that the second assertion follows, at once, from the first by means of a geometric rescaling argument. As a result, for the rest of this section we will tacitly assume to deal with some fixed value of the neck parameter (that is to say $c=2r_0$, in the notations above). It is instructive to analyze the specialization of this result to the Schwarzschild manifold.
            				
            				\begin{example} For $c>0$ fixed, let us consider on $\mathbb{R}^{n+1}\setminus \left\{|x|<m/2\right\}$ the Riemannian metric given in these coordinates by
            					\[
            					g^{(m)}_{ij}(x)=\left(1+\frac{m}{2|x|^{n-1}}\right)^{\frac{4}{n-1}}\delta_{ij}.
            					\]
            				Of course, the \textsl{leading term} of the metric is given in polar coordinates by
            					\[
            					h^{lead}_{ij}(r)=\frac{2m}{(n-1)r^{n-1}}\delta_{ij}
            					\]	
            					and thus it is readily checked that, for any $k\geq 0$, there exists a constant $C=C_{k,n}$ such that
            					\[
            					mC^{-1}_{k,n}\leq \sum_{i,j=1}^{n+1}\ \sum_{0\leq a\leq k} \ \sup_{r\geq c/2}  \left|\frac{d^a h_{ij}}{dr^a}\right| r^{n-1+a}\leq mC_{k,n}.
            					\]
            					Hence, the smallness condition required by Proposition \ref{pro:highdim} is equivalent to an inequality of the form
            					\[
            					m C_{k,n}<\varepsilon c^{n-1}
            					\] 
            					where $\varepsilon$ does not depend either on $m$ or $c$. As a result, for any fixed $c\geq \overline{c}$ we can perform the deformation of a Euclidean catenoid provided the mass parameter is small enough and, on the other hand, for any given value of the ADM mass the same conclusion holds true when perturbing catenoids of sufficiently large neck.
            				\end{example}	
            				
            				By virtue of the Implicit Function Theorem (namely by invoking e.\,g.\ Theorem 2.3 in \cite{AP95} or Theorem 2.7.2 in \cite{Nir74}), the proof of Proposition \ref{pro:highdim} amounts to checking that, for the range of values of the parameters $k$ and $q$ stated above, the linearized operator $D_2 H: \mathscr{C}^{k}_q(\R)\to \mathscr{C}^{k-2}_{q-2}(\R)$ is an isomorphism. This is a straightforward consequence of the following assertion.

            		\begin{lemma}\label{lem:leq0}
            			For any $k\geq 2$ and $-(n-2)<q<0$ the operator $T_0:\mathscr{C}^{k}_{q}(\mathbb{R})\to\mathscr{C}^{k-2}_{q}(\mathbb{R})$ is an isomorphism.	
            		\end{lemma}	
            		
            			\begin{proof}
            				First of all, one can ckeck that the functions $\left\{S^{1}_0(s), S^2_0(s) \right\}$ are linearly independent at each point $s\in\mathbb{R}$, for indeed the associated Wronskjan determinant is
            				\[
            				W_0(s)=\frac{n-1}{\cosh^{\frac{n-2}{n-1}}((n-1)s)}
            				\]
            				and thus all solutions to the ordinary differential equation $T_0 S= f_0$ can be written (by the method of variation of constants) in the form
            				\[
            				S(s)= S^{1}_{0}(s)\left[A-\int_{0}^{s} f_0(\s)W_{0}^{-1}(\s)S^2_0(\s)\,d\s \right]+S^2_0(s) \left[B+\int_0^{s}f_0(\s)W_0^{-1}(\s)S^1_0(\s)\,d\s\right]
            				\]
            				as $A, B$ vary in $\mathbb{R}$. We claim that, given $f_0\in\mathscr{C}^{k-2}_{q}(\mathbb{R})$ there is a unique choice of those parameters such that $S(s)\to 0$ as one lets $s\to\pm\infty$. To see this, one needs to rewrite the same solution in the (somewhat complicated) form
            				\begin{align*}
            				S(s) & =  A S^{1}_0(s)+B S^{2}_0(s) + \frac{S^{1}_0(s)}{n-1}\int_{0}^{+\infty}f_0(\s)(1+\omega(\s)\phi^{n-2}(\s)S^1_0(\s))\,d\s \\
            				& -\frac{S^1_0(s)}{n-1}\int_{s}^{+\infty}f_0(\s)\,d\s-\frac{\phi^{2-n}(s)}{n-1}\int_{0}^{s}f_0(\s)\phi^{n-2}(\s)S^1_0(\s)\,d\s \\
            				& -\frac{S^1_0(s)}{n-1}\int_{s}^{+\infty}f_0(\s)\omega(\s)\phi^{n-2}(\s)S^1_0(\s)\,d\s -\frac{S^1_0(s)}{n-1} \omega(s)\int_0^{s}f_0(\s)\phi^{n-2}(\s)S^1_0(\s)\,d\s
                        				\end{align*}
            				
            				where we have set (for the sake of brevity)
            				\[
            				\omega(s)=\int_{s}^{+\infty}\phi^{2-n}(\s)\,d\s.
            				\]
            				Such formula for $S(s)$ can be obtained, via some algebraic manipulations, observing that $S^2_0(s)=\psi(s)S^1_0(s)-1/\phi^{n-2}(s)$ and using the relation 
            				\[
            				\int_{\sigma_1}^{\sigma_2}\phi^{2-n}(\tau)\,d\tau = \omega(\sigma_1)-\omega(\sigma_2).
            				\]
            				At that stage, it is immediately apparent that $S(+\infty)=0$ if and only if 
            				\[
            				A+BT+I^{+}_f+\tilde{I}^{+}_f =0
            				\]
            				for
            				\[
            				T=\int_{0}^{+\infty}\phi^{2-n}(\s)\,d\s, \ \ I^{+}_f=\frac{1}{n-1}\int_{0}^{+\infty}f_0(\s)\,d\s, \ \ \tilde{I}^{+}_{f}=\frac{1}{n-1}\int_{0}^{+\infty}f_0(\s)\omega(\s)\phi^{n-2}(\s)S^1_0(\s)\,d\s
            				\]
            				which are finite under the assumption that $f_0\in\mathscr{C}^{k-2}_{q}(\mathbb{R})$ for $q<0$.
            				An analogous analysis for $s\to-\infty$ leads to the conclusion that $S(-\infty)=0$ if and only if a condition of the form $-A+BT+I^{-}_{f_0}+\tilde{I}^{-}_{f_0}=0$ holds, so that in the end the parameters $A$ and $B$ are the unique solution of $2\times 2$ linear system
            				\[
            				\begin{cases}
            				A+BT+I^{+}_{f_0}+\tilde{I}^{+}_{f_0} =0 \\
            				-A+BT+I^{-}_{f_0}+\tilde{I}^{-}_{f_0}=0
            				\end{cases}
            				\]
            				whose determinant equals $2T>0$.
            				On the other hand, once such choice is made the formulas provided above ensure that each term of $S(s)$ has a pointwise bound of the form $e^{-(n-2)|s|}$ or $e^{-2(n-1)|s|}$ or $e^{q|s|}$ so that, due to the assumption $-(n-2)<q<0$ the conclusion easily follows.
            			\end{proof}	
            			
            			\
            		
            					\begin{remark}
            						Notice that for $n=2$ the conclusion of Lemma \ref{lem:leq0} is \textsl{vacuous}, namely there is no value of $q$ for which $T$ is proven to be an isomorphism. In fact, we have seen in the first part of this article that an existence theorem cannot possibly hold true in ambient dimension three.
            					\end{remark}

            	\bibliographystyle{plain}

\begin{thebibliography}{HKW}
            		
            		\setcounter{footnote}{0}
            		
            		\bibitem[AP95]{AP95}\textsc{A. Ambrosetti, G. Prodi}, \textit{A primer of nonlinear analysis. Corrected reprint of the 1993 original}, Cambridge Studies in Advanced Mathematics, 34. Cambridge University Press, Cambridge, 1995. 
            		
            		\bibitem[ADM59]{ADM59}\textsc{R. Arnowitt, S. Deser, C. W. Misner}, \textit{Dynamical structure and definition of energy in general relativity}, Phys. Rev. (2) \textbf{116} 1959 1322-1330. 
            		
            		\bibitem[Bar86]{Bar86} \textsc{R. Bartnik}, \textit{The mass of an asymptotically flat manifold}, Comm. Pure Appl. Math. \textbf{39} (1986), no. 5, 661-693.
            		
            		\bibitem[BR15]{BR15}\textsc{Y. Bernard, T. Rivi\`ere}, \textit{Ends of Immersed Minimal and Willmore Surfaces in Asymptotically Flat Spaces}, preprint (arXiv:1508.01391).
            		
            		\bibitem[Car14]{Car14a}\textsc{A. Carlotto}, \textit{Rigidity of stable minimal hypersurfaces in asymptotically flat spaces}, Calc. Var. Partial Differential Equations  (\textsl{to appear}).
            		
            		%\bibitem[Car14b]{Car14b}\textsc{A. Carlotto}, \textit{Rigidity of stable marginally outer trapped surfaces in initial data sets}, Ann. Henri Poincar\'e (\textsl{to appear}).
            		
            		\bibitem[CCE15]{CCE15}\textsc{A. Carlotto, O. Chodosh, M. Eichmair}, \textit{Effective versions of the positive mass theorem}, Invent. Math. (\textsl{to appear}).
            		
            		\bibitem[CS14]{CS14} \textsc{A. Carlotto, R. Schoen}, \textit{Localizing solutions of the Einstein constraint equations}, Invent. Math. (\textsl{to appear}).
            		
            		%\bibitem[CDLM08]{CDLM08} \textsc{T. Colding, C. De Lellis, W. Minicozzi II}, \textit{Three circles theorems for Schr\"odinger operators on cylindrical ends and geometric applications}, Comm. Pure Appl. Math. \textbf{61} (2008), no. 11, 1540-1602.
            		
            		
            		
            		%\bibitem[CC81]{CC81}\textsc{Y. Choquet-Bruhat, D. Christodoulou}, \textit{Elliptic systems in $H_{s,\delta}$ spaces on manifolds which are Euclidean at infinity}, Acta Math. \textbf{146} (1981), no. 1-2, 129-150.
            		
            		\bibitem[EHLS11]{EHLS11}\textsc{M. Eichmair, L. H. Huang, D. A. Lee, R. Schoen}, \textit{The spacetime positive mass theorem in dimensions less than eight}, J. Eur. Math. Soc. (JEMS) \textbf{18} (2016), no. 1, 83-121.
            		
            		\bibitem[EM12]{EM12}\textsc{M. Eichmair, J. Metzger}, \textit{On large volume preserving stable CMC surfaces in initial data sets}, J. Differential Geom. \textbf{91} (2012), no. 1, 81-102.
            		
            		\bibitem[EM13a]{EM13a}\textsc{M. Eichmair, J. Metzger}, \textit{Large isoperimetric surfaces in initial data sets}, J. Differential Geom. \textbf{94} (2013), no. 1, 159-186. 
            		
            		\bibitem[EM13b]{EM13b}\textsc{M. Eichmair, J. Metzger}, \textit{Unique isoperimetric foliations of asymptotically flat manifolds in all dimensions}, Invent. Math. \textbf{194} (2013), no. 3, 591-630.
            		
            		%\bibitem[HM90]{HM90} \textsc{D. Hoffman, W. H. Meeks III}, \textit{The strong halfspace theorem for minimal surfaces}, Invent. Math. \textbf{101} (1990), no. 2, 373-377.
            		
            		%\bibitem[FS80]{FS80}\textsc{D. Fischer-Colbrie, R. Schoen} \textit{The structure of complete stable minimal surfaces in 3-manifolds of nonnegative scalar curvature}, Comm. Pure Appl. Math. \textbf{33} (1980), no. 2, 199-211. 
            		
            		%\bibitem[FC85]{FC85}\textsc{D. Fischer-Colbrie}, \textit{On complete minimal surfaces with finite Morse index in three-manifolds}, Invent. Math. \textbf{82} (1985), no. 1, 121-132.
            		
            		\bibitem[GM14]{GM14}\textsc{G. Galloway, P. Miao}, \textit{Variational and rigidity properties of static potentials}, Comm. Anal. Geom. (\textsl{to appear}).
            		
            		
            		\bibitem[HY96]{HY96}\textsc{G. Huisken, S. T. Yau},
            		\textit{Definition of center of mass for isolated physical systems and unique foliations by stable spheres with constant mean curvature}
            		Invent. Math. \textbf{124} (1996), no. 1-3, 281-311. 
            		
            		 \bibitem[LMcO85]{LMcO85}\textsc{R. Lockhart, R. McOwen}, \textit{Elliptic differential operators on noncompact manifolds}, Ann. Scuola Norm. Sup. Pisa Cl. Sci. (4) \textbf{12} (1985), no. 3, 409-447.
            		 
            		 %\bibitem[Maz91]{Maz91}\textsc{R. Mazzeo}, \textit{Elliptic theory of differential edge operators. I}, Comm. Partial Differential Equations \textbf{16} (1991), no. 10, 1615-1664.
            		 
            		 \bibitem[Mey63]{Mey63} \textsc{N. Meyers}, \textit{An expansion about infinity for solutions of linear elliptic equations}, J. Math. Mech. \textbf{12} (1963), 247-264.
            		 	
            		
            		\bibitem[MN12]{MN12} \textsc{A. Mondino, S. Nardulli}, \textit{Existence of Isoperimetric regions in non-compact Riemannian manifolds under Ricci or scalar curvature conditions}, Comm. Anal. Geom. \textbf{24} (2016), no. 1, 115--136. 
            		
            		\bibitem[MFN15]{MFN15}\textsc{A. Mu\~noz Flores, S. Nardulli}, \textit{The isoperimetric problem of a complete Riemannian manifold with a finite number of $C^{0}$-asymptotically Schwarzschild ends}, preprint (arXiv:1503.02361).
            		
            		\bibitem[Nir74]{Nir74}\textsc{L. Nirenberg}, \textit{Topics in nonlinear functional analysis. With a chapter by E. Zehnder.}, Lecture Notes, 1973-1974. Courant Institute of Mathematical Sciences, New York University, New York, 1974. viii+259 pp. 
            		
            		 %\bibitem[Pac13]{Pac13}\textsc{T. Pacini}, \textit{Desingularizing isolated conical singularities: uniform estimates via weighted Sobolev spaces}, Comm. Anal. Geom. \textbf{21} (2013), no. 1, 105-170.
            		
            		\bibitem[Sch83]{Sch83}\textsc{R. Schoen}, \textit{Uniqueness, symmetry, and embeddedness of minimal surfaces}, J. Differential Geom. \textbf{18} (1983), no. 4, 791-809. 
            		
            		%\bibitem[Sch89]{Sch89}\textsc{R. Schoen}, \textit{Variational theory for the total scalar curvature functional for Riemannian metrics and related topics. Topics in calculus of variations (Montecatini Terme, 1987)}, 120-154,  Lecture Notes in Math., vol. 1365, Springer, Berlin, 1989. 
            		
            		\bibitem[SY79]{SY79}\textsc{R. Schoen, S. T. Yau}, \textit{On the proof of the positive mass conjecture in general relativity}, Comm. Math. Phys. \textbf{65} (1979), no. 1, 45-76.
            		
            		%\bibitem[Sim83]{Sim83}\textsc{L. Simon}, \textit{Lectures on Geometric Measure Theory}, Centre for Mathematical Analysis (Australian National University), 1983.
            		
            		\bibitem[Wal84]{Wal84}\textsc{R. M. Wald}, \textit{General Relativity}, University of Chicago Press, Chicago, 1984.
            		
            		%\bibitem[Whi91]{Whi91}\textsc{B. White}, \textit{The space of minimal submanifolds for varying Riemannian metrics}, Indiana Univ. Math. J.  \textbf{40}  (1991),  no. 1, 161-200.
            		
            	\end{thebibliography}

            	\end{document}